\documentclass[12pt]{article}
\usepackage{amsmath,mathrsfs, eufrak, mathtools,amsthm}
\usepackage{amssymb}
\usepackage{newtxmath,newtxtext}
\usepackage{microtype}
\usepackage{titlesec}
\usepackage[top=1.0in,bottom=1.0in,left=1.5in,right=1in]{geometry}

\titleformat{\subsubsection}[runin]{\normalfont\bfseries}{\thesubsubsection}{1em}{}

\usepackage{pstricks,pst-plot,psfrag}
\usepackage{graphicx,subfigure,xspace,bm,tikz,pgfplots}
\usepackage[lined,linesnumbered,algoruled,noend]{algorithm2e}

\newtheorem{theorem}{Theorem}[section]
\newtheorem{definition}[theorem]{Definition}

\newtheorem{lemma}[theorem]{Lemma}

\newtheorem{remark}[theorem]{Remark}

\newtheorem{example}[theorem]{Example}
\newtheorem{corollary}[theorem]{Corollary}
\newtheorem{proposition}[theorem]{Proposition}  
\renewcommand{\qed}{\hfill $\blacksquare$}

\newcounter{para}[section]

\newcommand{\oprocendsymbol}{\hbox{$\bullet$}}
\newcommand{\oprocend}{\relax\ifmmode\else\unskip\hfill\fi\oprocendsymbol}

\newcommand{\g}{\mathfrak{g}}

\newcommand{\Sp}{\operatorname{Span}}
\newcommand{\R}{\mathbb{R}}

\renewcommand{\cal}{\mathcal}

\newcommand{\SE}{\operatorname{SE}}
\newcommand{\stab}{\operatorname{Stab}}
\newcommand{\so}{\mathfrak{so}}
\newcommand{\se}{\mathfrak{se}}

\newcommand{\Hess}{\operatorname{Hess}}
\newcommand{\tr}{\text{Tr}}

\renewcommand{\(}{\left (}
\renewcommand{\)}{\right )}

\newcommand{\F}{\mathbb{F}}

\definecolor{BBlue}{cmyk}{.98,0.10,0,.25}

\begin{document}

\title{\rule{\textwidth}{0.4mm} {\color{BBlue}\textsc{	 Controlling and Stabilizing a Rigid Formation using a few agents}}\\
\rule{\textwidth}{0.4mm}}
\date{}
\maketitle

\vspace{-2cm}
\begin{flushright}
{\footnotesize {\color{BBlue}{\bf Xudong Chen, M.-A. Belabbas, Tamer Ba\c sar}}}
\end{flushright}    

\newcommand{\f}{{f}}
\renewcommand{\F}{\bm{F}}
\newcommand{\h}{{h}}
\renewcommand{\g}{{g}}
\renewcommand{\u}{{u}}
\renewcommand{\H}{\bm{H}}
\newcommand{\of}{\omega_{\mathcal{F}}}

\begin{abstract}

We show in this paper that a small subset of agents of a formation of $n$ agents in Euclidean space can control the position and orientation of the entire formation. We consider here formations  tasked with maintaining inter-agent distances at prescribed values. It is known that when the inter-agent distances specified can be realized as the edges of a rigid graph, there is  a finite number of possible configurations of the agents that satisfy the distance constraints, up to rotations and translations of the entire formation. We show here that under mild conditions on the type of control used by the agents, a subset of them forming a clique  can work together to control the position and orientation of the formation as a whole.  Mathematically, we investigate the effect of certain allowable perturbations of a nominal dynamics of the formation system. In particular, we show that any such perturbation leads to a rigid motion of the entire formation. Furthermore, we show that the map which assigns to a perturbation the infinitesimal generator of the corresponding rigid motion  is locally surjective.


\end{abstract}

\section{ Introduction}
Formation control deals with the design of decentralized control laws to stabilize agents at prescribed distances from each other. It has been studied extensively in the past few years; see~\cite{mou2016undirected,cao2008control,krick2009,tabuada2005motion,summers2009formation} for various applications, and~\cite{OhAhnreviewAutomatica15} for a recent review of the extant work, and more references therein. It was shown in a recent paper~\cite{mou2016undirected} that when a pair of neighboring agents has a different understanding of what the target edge-lengths are, and when the graph describing the neighboring relationship is rigid~\cite{graver1993combinatorial,laman1970,zelazo2012rigidity}, the system undergoes a constant rigid motion. 
In this paper, we show that this property of the system, which was seen as a lack of robustness in earlier work, can be used to advantage, to control the orientation of the formation as a whole. To this end, we assume that some agents, linked by a few edges in the underlying graph,  can control the mismatches in target edge-lengths corresponding to these edges. The control then allows them to generate a rigid motion for the whole formation. Specifically, we show that if these agents form a non-degenerate triangle in the formation (or a non-degenerate $k$-simplex in the $k$-dimensional case), then they can in fact generate an {\it arbitrary} rigid motion. We then conclude from this fact that these agents can steer the whole formation arbitrarily close to any {\it desired} position and orientation. See Fig.~\ref{fig:form} for an illustration.

\begin{figure}[h]
\begin{center}
	\begin{tikzpicture}[ ,shorten >=1pt,auto,node distance=1.4cm,
  thin,main node/.style={circle,draw},inner sep=1pt]
\begin{scope}
[rotate=45]
 \node[main node,fill=blue!30] (1) at (0,0) {1};
  \node[main node,fill=blue!30] (4) at (-.71,-1)  {4};
  \node[main node,fill=blue!30] (2) at (.71,-1) {2};
\node[main node,fill=blue!30] (3) at (1.42,0) {3};
\node[main node,fill=blue!30] (5) at(2.1,-1) {5};
  \path[every node/.style={font=\sffamily\small}]
    (1)      edge   (2)
      edge   (3)
    (3)  edge (2)
    (1) edge(4)
    (4) edge (2)
    (3) edge (5)
    (2) edge(5)
        ;
     \node (cap) at (.71,-1.7) {$t=0$};         
\end{scope}
       
\begin{scope}[xshift=4.5cm,yshift=2cm,rotate=-0]
	  \node[main node,fill=red!30] (1) at (0,0) {1};
  \node[main node,fill=blue!30] (4) at (-.71,-1)  {4};
  \node[main node,fill=red!30] (2) at (.71,-.9) {2};
\node[main node,fill=red!30] (3) at (1.37,0) {3};
\node[main node,fill=blue!30] (5) at(2.1,-1) {5};
  \path[every node/.style={font=\sffamily\small}]
    (1)      edge[red,dashed]   (2)
      edge[red,dashed]   (3)
    (2)  edge[red,dashed] (3)
    (1) edge (4)
    (2) edge (5)
    (3) edge (5)
    (2) edge (4)
        ;
       \node (cap) at (.71,-1.6) {$t=5$};         
  
\end{scope}

\begin{scope}[xshift=9.5cm,yshift=1cm,rotate=-45]
	  \node[main node,fill=blue!30] (1) at (0,0) {1};
  \node[main node,fill=blue!30] (4) at (-.71,-1)  {4};
  \node[main node,fill=blue!30] (2) at (.71,-1) {2};
\node[main node,fill=blue!30] (3) at (1.42,0) {3};
\node[main node,fill=blue!30] (5) at(2.1,-1) {5};
  \path[every node/.style={font=\sffamily\small}]
    (1)      edge   (2)
      edge   (3)
    (3)  edge (2)
    (1) edge(4)
    (4) edge (2)
    (3) edge (5)
    (2) edge(5)
        ;
\node (cap) at (.71,-1.7) {$t=10$};         
\end{scope}
\end{tikzpicture}

\end{center}
	\caption{The rigid formation with 5 agents on the left ($T=0$) is as rest. In order to change its orientation by a rotation of 90 degrees and a translation to obtain to the formation on the right ($T=10$), we select nodes $\{1,2,3\}$, which form a triangle. By controlling, e.g., the target edge-length mismatch, we are able to control the position/orientation of the formation. The plain edges indicate that the target edge-lengths are met. The dashed edges indicate that there is a discrepancy between current edge-length and target edge-length. We show an intermediate step in the trajectory at $T=5$.}\label{fig:form}
\end{figure}
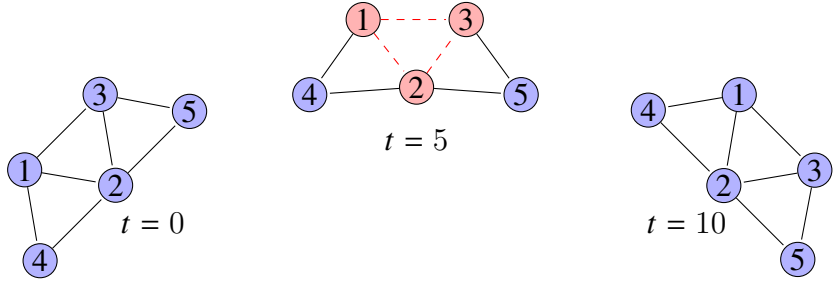

We now outline the contents of this paper more precisely. Let $G = (V,E)$ be an undirected graph on $n$ vertices with $V = \{1,\ldots,n\}$ the vertex set and $E$ the edge set. 
 Denote by ${\cal N}_i$ the neighbors of $i$ in $G$. Consider a (decentralized) formation with $n$ agents $x_1,\ldots, x_n\in \R^k$  following the dynamics:
\begin{equation}\label{eq:Model1}
\dot x_i = \sum_{j\in{\cal N}_i} f_{ij}(\|x_j - x_i \|) (x_j - x_i), \hspace{10pt} i=1,\ldots,n
\end{equation}
where 
each $f_{ij}: \R_{\ge 0} \longrightarrow \R$ is a continuously differentiable function modeling the interaction between $x_i$ and $x_j$.  A simple choice is $f_{ij}(x)=(x-\overline d_{ij})$, where $\overline d_{ij}$ is the {\em target value} for $\|x_i-x_j\|$, or the {\em target edge-length} for agents $x_i$ and $x_j$ to achieve and maintain. We call a configuration $q = (x_1,\ldots, x_n)$ {\em a target configuration} if $\|x_i - x_j\| = \overline d_{ij}$ for all $(i,j)\in E$. Note that if the underlying graph $G$ is rigid, then there are only {\em finitely many} target configurations modulo translations and rotations.  

It is well known that if the interactions among agents are reciprocal, i.e., $f_{ij} = f_{ji}$ for all $(i,j)\in E$, then system~\eqref{eq:Model1} is the gradient flow of the potential function:
\begin{equation}\label{eq:Potential}
\Phi(x_1,\ldots, x_n) := \sum_{(i,j)\in E} \int^{\|x_j - x_i \|}_{1} xf_{ij}(x) dx.
\end{equation}
Using this fact, one can easily generate locally stabilizing control laws. Indeed, for any target configuration~$q$, one can always stabilize $q$  by choosing the $f_{ij}$'s so that  $q$ is a local minimum point of $\Phi$. For example, if each $f_{ij}$ is monotonically increasing and has $\overline d_{ij}$ as its unique zero, then~$q$ is a global minimum point of~$\Phi$ (see, for example~\cite{doi:10.1137/15M105094X}).   
 However, we note that designing feedback control laws to stabilize {\em only} the target configurations is an open problem in general. A partial solution to the problem has been provided in~\cite{doi:10.1137/15M105094X} where we exhibited a class of rigid graphs in $\R^2$, termed triangulated Laman graphs, as well as a class of feedback control laws which stabilize  only the target configurations.    

The potential function $\Phi$ depends only the relative distances $\|x_i - x_j\|$ between agents, and hence is {\it invariant} under translations and/or rotations of the  state-space---that is invariant under an action of the group $\SE(k)$ of rigid motions. The control laws obtained as the gradient of such  potentials inherit their  invariance properties, and are called {\it $\SE(k)$-equivariant}.  These statements, made explicit below, imply that if $q$ is an equilibrium, then so is any configuration in its orbit $O_{q}$ under this group action. We thus conclude that an equivariant control law cannot be used to steer the system between two equilibrium configurations {\it within} the same orbit. We address in this paper the design of control laws that permit such steering. Moreover, in Section~\S\ref{sec:robustnessissue}, we indicate how such results can be leveraged to address the robustness issue in formation control that was pointed out in~\cite{mou2016undirected}.

We sketch below the idea of the proof. We know that a certain mismatch in target edge-lengths induces a rigid motion. We will show that, more generally, mismatches in {\it interaction laws} between neighboring agents produce rigid motions;  indeed,  mismatches in target edge-lengths simply belong to a finite-dimensional subset of the mismatches in interaction laws. We then introduce a function which assigns to a mismatch in interaction law  the corresponding infinitesimal generator of a rigid motion.  We then show that this function is locally {\it  surjective} as long as the edges that correspond to the mismatched interaction laws, together with the vertices incident to these edges,  form a $k$-simplex. We prove this fact by showing that the linearization of the map is of full rank. Said otherwise, if there are $(k+1)$ agents in the formation system who form a $k$-simplex in $\R^k$, and can control the mismatches in the interactions laws that correspond to the edges of the simplex, then these $(k+1)$ agents alone will be able to generate an {\it arbitrary} rigid motion for the formation. Using this fact, we show that we can generate a control law, which involves controlling the target edge-length mismatches, to reach a desired configuration within an arbitrarily small tolerance.

The remainder of the paper is organized as follows. In the following section, we introduce the basic definitions used in the paper, such as equivariant dynamical system, orbit through a point, and define the control models by which one is able to translate and/or rotate the entire formation via the method described above. We term any such control model {\it a formation system with a clique}. We also state the main controllability result precisely. In Section~\S\ref{sec:invariantorbit}, we present some general results about equivariant dynamical systems. We review the notion of an invariant orbit, and show how one can make sense of the Hessian associated with an invariant orbit. The signature of the Hessian, as similar to the case for a classical gradient system, determines the stability of an invariant orbit.  
The main technical content of the paper starts in Secection~\S\ref{sec:perturba}. The approach we use, as mentioned earlier, is to prove that all rigid motions of the formation can be obtained via target edge-length mismatches within a $k$-simplex. In order to do this, we start with a weaker statement. We show that if we consider arbitrary infinitesimal perturbations of the interaction control laws (these are perurbations in an infinite dimensional function space), then we can generate all rigid motions. This is in essence the content of Proposition~\ref{pro:linearspan} and Corollary~\ref{cor:submersion} in Section~\S\ref{sec:vertshift}. Said otherwise, we show that the infinite-dimensional space of perturbations of the interaction laws is mapped  surjectively to the space of infinitesimal rigid motions of formations. Based on this, we then show that there is a finite dimensional subspace of perturbations, namely the perturbations of interaction laws stemming from mismatches in edge-lengths, which generate all rigid motions.



\section{Controllability of Rigid Motions}\label{sec2}

\subsection{Preliminaries}

\subsubsection{Graphs, configurations and frameworks.}We denote by $P=\R^{nk}$ the {\bf configuration space} of a formation of $n$ agents in $\R^k$. We denote an element of $P$ by $p=(x_1,\ldots,x_n)$ where $x_i \in \R^k$ for all $i = 1,\ldots, n$. We call $p$ a configuration.  

The {\bf special Euclidean group} $\SE(k)$ (also known as the group of rigid motions) has a natural action over the configuration space. First, recall that each group element $\alpha$ of $\SE(k)$ can be represented by a pair  $(\theta, b)$ with $\theta$ a special orthogonal matrix and $b$ a vector in $\mathbb{R}^k$. With this representation, the multiplication of two group elements $\alpha_1= (\theta_1,b_1)$ and $\alpha_2 = (\theta_2,b_2)$ is given by
\begin{equation*}
\alpha_2\cdot \alpha_1 = (\theta_2 \,\theta_1,\, \theta_2 b_1 +  b_2).
\end{equation*}
A  group action of $\SE(k)$ on $P$ is given by  
\begin{equation}\label{eq:groupaction}
\alpha \cdot p := (\theta x_1 + b, \ldots, \theta x_n + b). 
\end{equation}
An element $\alpha \in SE(k)$ can also be thought of as an  {\it affine function} $\alpha:P\longrightarrow P$ sending $p$ to $\alpha \cdot p$. 

We denote by $O_p$ the {\bf orbit of $p$} for the group action of $\SE(k)$ on $P$, i.e., 
$$
O_p:= \left\{ \alpha\cdot p \mid  \alpha\in \SE(k)\right\}.
$$ 
Thus, $O_p$ is the set of configurations that are related by a rigid motion. The {\bf stabilizer} of $p$, denoted by $\stab(p)$, is a subgroup of $\SE(k)$ defined as the set of group elements that leave $p$ fixed through the above action: 
$$
\stab(p) := \left \{ \alpha \in \SE(k) \mid \alpha\cdot p = p \right \}. 
$$   
A subgroup of $\SE(k)$ is said to be {\bf trivial} if it contains {\em only} the identity element. 
If $\stab(p)$ is a trivial subgroup, then it is clear that $O_p$ is diffeomorphic to $\SE(k)$, which we denote by $O_p \approx \SE(k)$. 
We now introduce the notion of rank of a configuration:

\begin{definition}[Rank of a configuration]
Let $p = (x_1,\ldots, x_n)$, with $x_i\in \R^k$, be a configuration. The {\bf rank} of $p$ is the dimension of the span  of the vectors $\{x_2 - x_1,\ldots, x_n - x_1\}$.  If the rank of $p$ is $k$, then $p$ is of {\bf full rank}. 
\end{definition}

We note here that the rank of the configuration $p$ is independent of which vector~$x_i$ ($x_1$ in the above) is subtracted from the others. The rank of a configuration $p$ is known~\cite{chen2015controllability} to be the least dimension of a subspace of $\R^k$ in which $p$ can be embedded. With the definition above, we state the following result: 

\begin{lemma}\label{lem:trivialstabilizer}
If a configuration $p$ is of full-rank, then the stabilizer $\stab(p)$ is trivial, and hence $O_p \approx \SE(k)$. 
\end{lemma}

\begin{proof}
A group element $\alpha = (\theta, b)\in \SE(k)$ can be represented by a matrix  as follows:
$$A = 
\begin{bmatrix}
\theta & b \\
0 & 1
\end{bmatrix}\in  \R^{(k+1)\times (k+1)}.
$$
The group multiplication is then the matrix multiplication. Since $p$ is of full-rank, without loss of generality, we can take the linear span of $\{x_2 - x_1,\ldots, x_{k+1} - x_1\}$ to be $\R^k$. Define the matrix $X\in\R^{(k+1)\times (k+1)}$ by setting its  $i$-th column to $(x_i,1)\in \R^{k+1}$. The matrix $X$ is not of full rank only if one of its columns (say, without loss of generality, the first) is a linear combination of the others: $\sum_{i=2}^{n} c_i (x_i,1)=(x_1,1)$ for some constants $c_i$'s. The equality of the last coordinates imposes that $\sum_{i=2}^{n} c_i =1$. Ignoring the last coordinate, we also have $\sum_{i=2}^{n} c_i x_i = x_1= \sum_{i=2}^{n} c_i x_1$ and thus $\sum_{i=2}^{k+1} c_i (x_i-x_1)=0$, which is a contradiction. The matrix $X$ is thus nonsingular.  Now, let $\alpha = (\theta, b)$ be in $\stab(p)$. Then, 
$
A X = X
$.   
Since $X$ is nonsingular, we conclude that $A = I$. 
\end{proof}

Let $G=(V,E)$ be an undirected graph with $n$ vertices, a configuration $p = (x_1,\ldots, x_n)$ can be viewed as an embedding of $G$ in $\R^n$ by assigning to vertex~$i$ the coordinate $x_i$. We call the pair $(G,p)$ a {\bf framework}. 
Let $G' = (V',E')$ be the subgraph of $G$ induced by $V'$. We denote by $(G',p')$, with $p'\in \R^{k |V'|}$,  the associated sub-framework. We call $p'$ the {\bf sub-configuration} associated with $G'$. 
Recall that a {\bf clique} of $G$ is a set of vertices $V'$ such that the induced subgraph $G' = (V',E')$ is a complete graph. We now extend the definition to a framework:

\begin{definition}[Framework with clique]
Let $G$ be a graph and $p \in \R^{nk}$. A framework $(G,p)$ is {with a clique} if there is a subgraph $G'$ of $G$ such that $G'$ is a complete graph with $(k+1)$ vertices. The associated sub-framework $(G',p')$ is a {\bf clique} of $(G,p)$.
 \end{definition}

 If, in addition, the sub-configuration $p'$ is of full rank, then $(G,p)$ is a framework with a full rank clique $(G',p')$.  For example, in the case $k=2$, a full-rank clique is a nondegenerate triangle, and in the three-dimensional case,   a nondegenerate tetrahedron.  
We also note that rigid transformations preserve linear independence of vectors. Thus, if a framework $(G,p)$ has a (full-rank) clique $(G',p')$, then clearly $(G,\, \alpha\cdot p)$ has a (full-rank) clique $(G',\alpha\cdot p')$ for any $\alpha\in \SE(k)$, 

\subsubsection{Equivariant systems. } We briefly review basic definitions regarding equivariant systems, and describe some of their properties. 
We start with the following definition: 
 
\begin{definition}[Equivariant system]\label{def:equivariantsys}
Consider an arbitrary dynamical system 
$
\dot p = f(p)
$ in a Euclidean space $P$. Let $\cal{A}$ be a Lie group acting smoothly on $P$. Denote by $d\alpha_p\,\alpha: T_pP\longrightarrow  T_pP$  the derivative of $\alpha$ at $x$.  
Then, the dynamical system is an {\bf  $\cal{A}$-equivariant system} if for any element $\alpha\in \cal{A}$ and any $p\in P$, we have
\begin{equation}\label{eq:vectorfieldrelation}
f(\alpha\cdot p) = d\alpha_p (f(p)).
\end{equation}
 Equivalently, if we let $\phi_t(p_0)$ be the solution of $\dot p=f(p)$ with initial condition $p_0$, then 
$$
\phi_t(\alpha\cdot p) = \alpha \cdot \phi_t(p),
$$
for all $p\in P$,  $\alpha\in \SE(k)$, and $t\ge 0$. 
\end{definition}


In our case, the special Euclidean group acts on the configuration space $P$ as an affine transformation~\eqref{eq:groupaction}. 
We show  that the formation system~\eqref{eq:Model1} is  an $\SE(k)$-equivariant system with respect to the group action introduced above. To this end, set  $y = (y_1,\ldots, y_n)\in \R^{kn}$, with $y_i\in \R^k$. From~\eqref{eq:groupaction}, the derivative $d\alpha_p$, at any $p\in P$ is the  linear map
$$
d\alpha_p (y) = (\theta y_1,\ldots, \theta y_n). 
$$
Now,  let $f(p) = (f_1(p),\ldots, f_n(p))$ be the vector field of system~\eqref{eq:Model1} at $p$, where 
 \begin{equation}\label{eq:vectorfield}
 f_i(p):= \sum_{v_j\in {\cal N}_i} f_{ij}(\| x_j - x_i \|)  (x_j - x _i), 
 \end{equation} 
is the $i$-th component of $f(p)$.  
We have the following fact:

\begin{lemma}\label{lem: equisys} 
System \eqref{eq:Model1} is $\SE(k)$-equivariant with respect to the group action~\eqref{eq:groupaction}. 
\end{lemma}

\begin{proof}
Let $\alpha = (\theta, b)\in \SE(k)$;  it suffices to show that 
$$
f_i(\alpha\cdot p) = \theta f_i(p), \hspace{10pt} \forall\, i=1,\ldots. n.
$$
Note that 
$$
\| (\theta x_j + b) - (\theta x_i + b)  \| = \|x_j - x_i\|,
$$
and hence, 
$$
f_i(\alpha \cdot p) = \theta \sum_{j\in {\cal N}_i} f_{ij}(\| x_j - x_i\|) (x_j - x_i) = \theta f_i(p).
$$\,
\end{proof}


We now investigate equilibrium points of equivariant dynamics. Let $p$ be  an equilibrium of system~\eqref{eq:Model1}, i.e., $f(p) = 0$.  Then, it is immediate from~\eqref{eq:vectorfieldrelation} that $f(\alpha\cdot p) = 0$ for any $\alpha \in \SE(k)$. In other words, any configuration in the orbit $O_p$ is an equilibrium of~\eqref{eq:Model1}. Since~\eqref{eq:Model1} is a gradient system with $\Phi$ the potential function (defined in~\eqref{eq:Potential}), an equilibrium of~\eqref{eq:Model1} is a {\bf critical point} of $\Phi$. We can thus call $O_p$ {\bf a critical orbit} of $\Phi$.  
Furthermore, let $H_p$ be the {\bf Hessian} of $\Phi$ at $p$: 
\begin{equation}\label{eq:Hess1}
H_p:= \frac{\partial^2 \Phi(p)}{\partial p^2}
\end{equation}
The following lemma presents  well-known facts about the Hessian matrix $H_p$: 

\begin{lemma}\label{lem:Hessianinv}
Let $\Phi:P \longrightarrow \R$ be  a function invariant  under a Lie-group action over a Euclidean space. Denote by $m$ the dimension of a critical orbit ${O}_p$  and  by  $H_p$ the Hessian of $\Phi$ at $p$.  Then for any $p' \in O_p$, the  Hessian matrices $H_{p'}$ and $H_p$ are related by a similarity transformation. In particular, $H_{p'}$ and $H_p$ have the same eigenvalues.    In addition,  $H_{p}$  has at least $m$ zero eigenvalues, and the null space of $H_{p}$ contains the tangent space of $O_p$ at $p$. 
\end{lemma}


It follows from Lemma~\ref{lem:Hessianinv} that $H_p$ and $H_{p'}$ have the same number of zero eigenvalues if $p$ and $p'$ belong to the same critical orbit, and moreover,  $\dim O_p$ is the 
least number of zero eigenvalues of the Hessian matrix at~$p$. Following Lemma~\ref{lem:Hessianinv}, we have the following definition: 

\begin{definition}\label{def:firstdefforexp}
A critical orbit $O_p$ of the potential function $\Phi$ (defined in~\eqref{eq:Potential}) is {\bf nondegenerate} if there are exactly $\dim O_p$ zero eigenvalues of the Hessian matrix $H_p$. If all the other eigenvalues are positive (note that $H_p$ is symmetric), then $O_p$ is {\bf exponentially stable}.   
\end{definition}

\subsection{Formation system with a clique and  $\epsilon$-controllability of rigid motions}\label{Frameworkwithclique}
Let $q$ be a target configuration. We assume for the remainder of the paper that the framework $(G, q)$ has a {\em full rank clique} $(G^*,q^*)$. For a vertex~$i$ of the complete graph $G^*$, we let $\cal{N}^*_i$ be the set of neighbors of~$i$ in $G^*$.   
Recall that for a configuration $p$, we have used $f(p) = (f_1(p),\ldots, f_n(p))$, with $f_i(p) \in \R^k$, the vector field of~\eqref{eq:Model1} at $p$.  Now, consider the following  modified formation control model (compared with~\eqref{eq:Model1}): 
 \begin{equation}\label{eq:ControlModel}
\dot {x}_i = 
\left\{
\begin{array}{ll}
f_i(p)+ \sum_{j \in \cal{N}^*_i } u_{ij}(t) (x_j - x_i)   & \text{ if } i\in V^* \\
f_i(p) & \text{ otherwise},
\end{array} \right.
\end{equation}
where each $u_{ij}$ is a scalar control. We note that the control inputs $u_{ij}$'s are not necessarily reciprocal, i.e., we do {\em not} require $u_{ij} = u_{ji}$. Thus, the multi-agent system~\eqref{eq:ControlModel} is controlled via the interactions among the agents in the sub-configuration $q^*$. 
The total number of control inputs is thus $k(k+1)$, which is relatively small compared to the total number of agents (which is $n$).  For example, in the two (resp. three) dimensional case, we have that the number of control inputs is $6$ (resp. $12$), but the number $n$ of agents can be arbitrarily large.  
We simply let $u := (u_{ij})$ be the ensemble of the controls $u_{ij}$'s, and $u[0,T]$ be the control~$u$ over the  interval $[0,T]$.  We call~\eqref{eq:ControlModel} a {\bf formation system with a clique}.  

 
We now formalize the notion of controllability used in this paper. Let $P$ be an arbitrary Euclidean space, and $M$ be a smooth submanifold of $P$.  We say that $M$ is {\bf path-connected} if for any two points $p_0,\ p_1\in M$, there is a continuous curve $p: [0,1]\longrightarrow M$, with $p(0) = p_0$ and $p(1) = p_1$. In our case, $P$ is the configuration space and $M$ is $O_q$. 
Since $q$ is of full rank (as $q^*$ is), from Lemma~\ref{lem: equisys}, we have $O_q \approx \SE(k)$, and hence $O_q$ is path-connected. 
Next, we let $p$ be a configuration in $P$; we define the distance between $p$ and the orbit $O_{q}$ as
$$
d(p, O_q) := \inf\left \{\| p - q'\| \mid  q'\in O_q  \right \}.
$$  
We say that $p$ is {\bf $\bm{\epsilon}$-close to $O_q$} if $d(p,O_q) < \epsilon$. 
We now introduce $\epsilon$-controllability:

\begin{definition}[$\epsilon$-controllability] Let $P$ be an arbitrary Euclidean space, and $M$ be a path-connected smooth submanifold of $P$. A control system $\dot p = f(p,u)$, defined over $P$, is {\bf $\epsilon$-controllable over $M$} if for any two points $p_0, p_1$ in $M$, and any error tolerance $\epsilon>0$, there is a time $T>0$ and an admissible control $u[0,T]$ such that the solution $p(t)$,  from the initial condition $p_0$ and with the control $u[0,T]$, is $\epsilon$-close to $M$ for all $t\in [0,T]$, and moreover, $\|p(T) - p_1 \| < \epsilon$.   
\end{definition}
 
We now state the first main result of the paper.

\begin{theorem}\label{thm:Main}
Let $G$ be a  rigid graph, and $(G,q)$ be a framework with a full-rank clique. If ${O}_{q}$ is a critical orbit of $\Phi$ (defined in~\eqref{eq:Potential}) and is exponentially stable,  then system~\eqref{eq:ControlModel} is $\epsilon$-controllable over $O_q$. Moreover, for any two configurations $q_0$ and $q_1$ in $O_q$, there is a {\em constant} control law $u[0,T]$ that steers the system from $q_0$ to a configuration within an  $\epsilon$-ball of $q_1$ in $P$. 
\end{theorem} 


 
\section{Invariant Orbits and Local Perturbations}

\subsection{Invariant orbit}\label{sec:invariantorbit}
In this subsection, we derive  some relevant properties of invariant orbits of equivariant dynamical systems. Recall that for an arbitrary dynamical system $\dot p = f(p)$,  we have introduced $\phi_t(p)$ to denote the solution at time $t \geq 0$ with  initial condition~$p$.   We start with the definition of an $f$-invariant orbit: 

\begin{definition}[$f$-invariant orbit] Let $\cal{A}$ be a Lie group acting on a Euclidean space $P$. Let $\dot p = f(p)$ be an $\cal{A}$-equivariant system defined over $P$. The orbit $O_p= \cal{A} \cdot p$ is said to be {\bf $f$-invariant} if for any initial condition $p'\in O_p$, the trajectory $\phi_t(p')$  remains in $O_p$: $p' \in O_p \Rightarrow \phi_t(p') \in O_p$ for all $t \ge 0$
\end{definition}

Note that if $f$ is the gradient of  an $\cal{A}$-invariant potential function $\Phi$, then a critical orbit $O_p$ of $\Phi$ is $f$-invariant: indeed, we have that $f(p') = 0$ for all $p'\in O_p$, and hence $\phi_t(p') = p'$ for all $t\ge 0$ and all $p'\in O_p$. Conversely, any invariant orbit of $f$ has to be a critical orbit of $\Phi$. 

The interaction laws considered in~\eqref{eq:Model1} were always reciprocal, which led to the fact that~\eqref{eq:Model1} is a gradient system. The control methodology we propose in this paper, however, requires to break such a reciprocity, as we saw in~\eqref{eq:ControlModel}.  In the remainder of this subsection, we thus {relax} the condition imposed on~\eqref{eq:Model1} that the interactions between neighboring agents are reciprocal and study the  system
\begin{equation}\label{eq:Model2}
\dot x_i = \sum_{j\in V_i} f_{ij}(\|x_j - x_i\|)(x_j -x_i),
\end{equation} 
which is similar to~\eqref{eq:Model1}, but {\it without} the requirement that $f_{ij}$ is equal to $f_{ji}$. System~\eqref{eq:Model2} can be shown, proceeding as in Lemma~\ref{lem: equisys}, to be an $\SE(k)$-equivariant system, but it does not need to be a gradient system. 


We now state some basic facts about the vector field $f$ when restricted to  $O_p$. Denote by $\mathfrak{so}(k)$ the vector space of $k$-by-$k$ skew-symmetric matrices, that is the { \it Lie algebra}~\cite{helgason2001differential} associated with the special orthogonal group ${\rm SO}(k)$.  Denote by  $\exp(\cdot)$ the matrix exponential. Then, $\exp(\cdot)$ maps $\so(k)$ {\it onto} the group ${\rm SO}(k)$. For a matrix $\Omega\in \mathfrak{so}(k)$ and a real number $t$, we adopt the notation used in~\cite{Helmke_MSA_14} and define
$$
\frac{\exp(\Omega t) - I}{\Omega} := It + \frac{\Omega}{2!}t^2 + \frac{\Omega^2}{3!}t^3 + \cdots.
$$   
The expression above is well defined for all $\Omega \in \so(k)$ even if $\Omega = 0$. 
We further define $\se(k) := \so(k)\times \R^k$; it is the Lie algebra of the special Euclidean group $\SE(k)$. Let $T_{p'}O_{p'}$ be the tangent space of $O_{p}$ at $p'$ for $p'\in O_p$.  Recall that a group element $(\theta, b)\in \SE(k)$ can be represented by a matrix  as follows:
$$
\begin{bmatrix}
\theta & b \\
0 & 1
\end{bmatrix}\in  \R^{(k+1)\times (k+1)}.
$$
Similarly, we can represent an element  $ (\Omega, v) \in \se(k)$ by 
$$
\begin{bmatrix}
\Omega & v \\
0 & 0
\end{bmatrix} \in \R^{(k+1)\times (k+1)}.
$$
By computation, the matrix exponential map $\exp: \se(k) \to \SE(k)$ is given by
\begin{equation}\label{eq:matrixexponential}
\exp\left (
\begin{bmatrix}
\Omega & v \\
0 & 0
\end{bmatrix}
 \right ) = 
 \begin{bmatrix}
\exp(\Omega) &  \frac{\exp(\Omega) - I}{\Omega} \ v \\
0 & 1
\end{bmatrix}.
\end{equation}
This exponential map is also surjective~\cite{helgason2001differential}. 
We next have the following fact (see~\cite{Helmke_MSA_14} for a similar result):

\begin{proposition}\label{pro:flowoninvorb}
Let $p=(x_1,\ldots,x_n)$ be a full-rank configuration, and ${O}_p$ be an $f$-invariant orbit of system \eqref{eq:Model2}. Then, the following three properties hold:
\begin{enumerate}
\item[1.] For each $p'=(x_1',\ldots,x_n') \in O_p$, the vector field $f(p')$ lies in $T_{p'}O_{p}$. In particular, there is a unique element $(\Omega_{p'}, v_{p'}) \in \se(k)$ such that
\begin{equation}\label{eq:invariantorbit1}
f_i(p') = \Omega_{p'} x_i' + v_{p'}, \hspace{10pt} \forall i\in V. 
\end{equation}
\item[2.] If $p' = \alpha\cdot p$ for $\alpha = (\theta, b) \in \SE(k)$, then
\begin{equation}\label{eq:invariantorbit2}
\left\{
\begin{array}{l}
\Omega_{p'} = \theta\, \Omega_{p} \, \theta^\top, \\
v_{p'} = \theta\, v_{p} - \theta\, \Omega_{p} \, \theta^\top  b.
\end{array}
\right. 
\end{equation} 
\item[3.]  The solution $p'(t) = (x'_1(t),\ldots, x'_n(t))$ for an initial condition $p' = (x'_1,\ldots, x'_n)\in O_p$ is given by 
 $$p'(t) = \exp\((\Omega_{p'}, v_{p'})t\)\cdot p'.$$
 More specifically, using~\eqref{eq:matrixexponential}, we obtain that
\begin{equation}\label{eq:soloninvariantrobit}
x'_i(t) = \exp(\Omega_{p'}t) \ x'_i \ + \    \frac{\exp(\Omega_{p'} t) - I}{\Omega_{p'}} \ v_{p'}. 
\end{equation}
\end{enumerate}
\,
\end{proposition}

\begin{proof}
The first item follows from the fact that the tangent space of $O_p$ at $p' = (x'_1,\ldots, x'_n)$ is given by
$$
T_{p'}O_p = \left \{(\Omega x'_1+ v,\ldots, \Omega x'_n + v)  \mid (\Omega,v) \in \se(k) \right\}.
$$ 
The second item follows from the fact that system~\eqref{eq:Model1} is $\SE(k)$-equivariant: if $p' = \alpha\cdot p$ with $\alpha = (\theta, b)$ 
$$
f_i(p') = \theta f_i(p), \hspace{10pt} \forall\, i\in V,
$$ 
which, recalling the definition of the group action~\eqref{eq:groupaction}, yields~\eqref{eq:invariantorbit2}. The third item directly follows computation; indeed we differentiate~\eqref{eq:soloninvariantrobit} with respect to~$t$ and obtain that
$$
\dot x'_i(t) = \(\theta\, \Omega_{p'} \, \theta^\top\) x'_i(t) + \(\theta\, v_{p'} - \theta\, \Omega_{p'} \, \theta^\top  b\),
$$
with $(\theta, b)$ given by
$$
\theta := \exp(\Omega_{p'} t) \hspace{5pt} \text{ and } \hspace{5pt} b :=  \frac{\exp(\Omega_{p'} t) - I}{\Omega_{p'}} \ v_{p'}. 
$$ 
On the other hand, we know from~\eqref{eq:invariantorbit2} that $(\Omega_{p'(t)}, v_{p'(t)})$ and $(\Omega_{p'}, v_{p'})$ are related by 
\begin{equation*}
\left\{
\begin{array}{l}
\Omega_{p'(t)} = \theta\, \Omega_{p'} \, \theta^\top, \\
v_{p'(t)} = \theta\, v_{p'} - \theta\, \Omega_{p'} \, \theta^\top  b,
\end{array}
\right. 
\end{equation*} 
One thus have that
$$
\dot x'_i(t) = \Omega_{p'(t)} x'_i(t) + v_{p'(t)},
$$ 
which completes the proof. 
\end{proof}


From Proposition~\ref{pro:flowoninvorb}, the vector field over an $f$-invariant orbit $O_p$, for $p$ a full-rank configuration, is entirely determined by its value $f(p)$ at a single configuration~$p$. Moreover, $f(p)$ can be uniquely represented by an element $(\Omega_p,v_p)\in \se(k)$.  Thus, with a slight abuse of notation, we denote by 
\begin{equation}\label{eq:notation1}
(\Omega_p,v_p)\cdot p:=\left( \Omega_p x_1 + v_p,\ldots, \Omega_p x_n + v_p\right), 
\end{equation} 
the vector field $f$ at $p$. 

We have argued at the beginning of this subsection that if $f_{ij} = f_{ji}$, then an $f$-invariant orbit has to be a critical orbit. We have also defined in~\eqref{eq:Hess1} the Hessian matrix at a point of a critical (and hence, an $f$-invariant) orbit. The eigenvalues of the Hessian matrix at any such point of the orbit tell whether  the orbit is  exponentially stable or not. 
In order to study the stability properties of an invariant orbit of a generalized formation system~\eqref{eq:Model2}, we need to extend the definition of ``Hessian'' to non-gradient systems. To do so, we introduce  a matrix which agrees with the original definition~\eqref{eq:Hess1} if the interactions among neighboring agents are reciprocal. The definition can be shown to be equivalent to a definition introduced in the seminal paper~\cite{field1980equivariant}. By a slight abuse of terminology, we refer to this matrix as the Hessian matrix as well. 
Recall that from Proposition~\ref{pro:flowoninvorb}, if $O_p$ is an invariant orbit and $p$ is a full-rank configuration, there exists a unique $(\Omega, v)\in \se(k)$ so that  $f(p) = (\Omega, v)\cdot p$.  Define an auxiliary vector field $h$ on $P$ as follows: for a configuration $p' = (x'_1,\ldots, x'_n)$,  let
\begin{equation}
\label{eq:defauxilh}
h(p') := (\Omega x'_1 + v,\ldots, \Omega x'_n + v).  
\end{equation}
We then have the following definition:  
 
\begin{definition}[Hessian at an $f$-invariant orbit] Let $p$ be a full rank configuration,  $O_p$ be an $f$-invariant orbit and $(\Omega, v)\in \se(k)$ so that  $f(p) = (\Omega, v)\cdot p$. Let $h$ be as in~\eqref{eq:defauxilh} and set $$\widehat f := f-h.$$
We define the Hessian at $p$ to be the negative of the Jacobian of $\widehat f$ at $p$, i.e., 
\begin{equation}\label{eq:Hess2}
\Hess(p) := -\frac{\partial \widehat f(p)}{\partial p}.  
\end{equation}\,
\end{definition}
When $f$ is the gradient of the potential $\Phi$, then $O_p$ is a critical orbit and thus $(\Omega, v)=(0,0)$ and  the auxiliary vector field $h$ vanishes everywhere. i.e., $h \equiv 0$. Thus,  \eqref{eq:Hess2} coincides with~\eqref{eq:Hess1}. 
We refer to \cite{field1980equivariant} (Section 3, Proposition~J.) for the definition of Hessian under a more general context. 
Similarly, we have the following fact for generalized Hessian matrices (compared to Lemma~\ref{lem:Hessianinv}):

\begin{lemma}\label{lem:secondfactHess}
Let $O_p$ be an $f$-invariant orbit of system~\eqref{eq:Model2} with $p$ a full-rank configuration. Then, for any $p'\in O_p$, the two Hessian matrices $\Hess(p')$ and $\Hess(p)$ are related by a similarity transformation. There are at least $k(k+1)/2$ zero eigenvalues of $\Hess(p)$. The null space of $\Hess(p)$ contains the tangent space of $O_p$ at $p$.  
\end{lemma}

Following Lemma~\ref{lem:secondfactHess}, we can generalize Definition~\ref{def:firstdefforexp} to an $f$-invariant orbit: 

\begin{definition}\label{def:seconddefforexp}
Let $O_p$ be an $f$-invariant orbit of system~\eqref{eq:Model2}, with $p$ a full rank configuration. The orbit $O_p$ is {\bf nondegenerate} if the Hessian $\Hess(p)$ has exactly $k(k+1)/2$ zero eigenvalues. If all the other eigenvalues of $\Hess(p)$ have positive real parts, then $O_p$ is {\bf exponentially stable}.   
\end{definition}

Recall that if an $f$-invariant orbit $O_p$ is exponentially stable, then there is an open neighborhood $U$ of $O_p$ such that any solution of~\eqref{eq:Model2}, with  the initial condition $p(0)\in U$,  converges to $O_p$ exponentially fast. Moreover, the open neighborhood $U$ can be chosen such that it is  $\SE(k)$-invariant. 
We can in fact characterize the behavior of the solution near an invariant orbit in stronger terms.  
To do so, we  introduce the following definition:

\begin{definition}
Let $O_p$ be an $f$-invariant orbit of system~\eqref{eq:Model2}. Let $p'$ be a point in $O_p$. The {\bf stable manifold} of $p'$ under $f$ is the differentiable manifold given by
$$
W^s(p'):= \left \{  p'' \in P  \mid  \lim_{t\to \infty} \|\phi_t(p'') - \phi_t(p)'\| = 0   \right\}.
$$
For an open neighborhood $U$ of $p'$, the {\bf local stable manifold} is 
$$
W^s_U(p') := W^s(p') \cap U. 
$$\,
\end{definition}
We emphasize  that this definition takes into account the fact that $p'$ is \emph{not} necessarily an equilibrium point, but belongs to an invariant orbit which is attractive. If $p'$ is moreover an equilibrium, then $\phi_t(p')=p'$, and the above definition then reduces to the usual definition of a stable manifold. For example, the two-dimensional dynamical system $\dot x = x; \dot y = -y$ has the $x$-axis as an attractive invariant subspace of $\R^2$. For a point $p' = (1,0)$ on the $x$-axis, its stable manifold is given by the vertical straight line $\{(1,y) \mid y \in \R\}$. Indeed,  we have $\phi_t(p')=(e^{t},0)$ and for  $p''=(1,y)$, we have  $\phi_t(p'')=(e^t,ye^{-t})$, from which we conclude that  $\lim_{t\to \infty} \|\phi_t(p'') - \phi_t(p')\| = 0$ as required.

The following lemma then presents some well known facts about stable manifolds of points in an exponentially stable $f$-invariant orbit (not necessarily comprised of equilibrium points):   


\begin{lemma}\label{lem:fibersmoothsubmfld}
Let $O_p$ be an exponentially stable $f$-invariant orbit. Then, there is an $\SE(k)$-invariant open neighborhood $U$ of $O_p$ such that the following three properties hold:
\begin{enumerate}
\item[1.] The local stable manifold $W^s_U(p)$ intersects $O_p$ transversally at $p$.  
\item[2.] For any $\alpha\in \SE(k)$, we have 
\begin{equation*}\label{eq:fiberequivariant}
W^s_U(\alpha\cdot p) = \alpha\cdot W^s_U(p).
\end{equation*}
\item[3.] For any $p'\in O_p$ with $p'\neq p$, we have 
\begin{equation*}\label{eq:fiberdisjoint}
 W^s_U(p') \cap W^s_U(p) = \varnothing. 
\end{equation*}
\end{enumerate}\,
\end{lemma}   

\begin{remark}\label{rmk:definitionofNs}
When $O_p$ is exponentially stable (and hence nondegenerate), the null space of the Hessian matrix $\Hess(p)$ is $T_pO_p$. Denote by $N^s(p)$ the range space of $\Hess(p)$. Then, $N^s(p)$ is invariant under $\Hess(p)$. The eigenvalues of $\Hess(p)$, when restricted to $N^s(p)$, have positive real parts. We also have the following relation:
$$
T_pO_p \oplus N^s(p) = \R^{kn}.
$$
Furthermore, it is known~\cite{field1980equivariant} that the tangent space of the stable manifold $W^s(p)$ at $p$ is given by $N^s(p)$. It then follows that $W^s_U(p)$ intersects $O_p$ transversally at $p$. 
\end{remark}

In the remainder of the paper, we  write $\phi_t(p; f)$, $\Hess(p; f)$, and $W^s(p; f)$ to indicate  that the trajectory, the Hessian matrix, and the stable manifold depend on the vector field of system~\eqref{eq:Model2}.

\subsection{Local perturbation lemma}\label{sec:perturba}
We  investigate in this subsection the behavior of an exponentially stable $f$-invariant orbit $O_p$ under a small perturbation of the vector field $f$.  Recall that exponentially stable zeros of a vector field are ``robust'' under small perturbations of the vector field. That is, we consider  $\dot p = f(p)$ defined over an Euclidean space $P$ with  $p_0$  an (isolated) exponentially stable equilibrium point. Then,  for any small perturbation of $f$, there exists a unique equilibrium point $p'_0$ within an open neighborhood of $p_0$. Moreover, the equilibrium point $p'_0$ is exponentially stable. Our objective in this subsection is to obtain a similar result for an $f$-invariant orbit under an equivariant dynamics.


We  consider the class of generalized formation systems described by~\eqref{eq:Model2}. We find it useful to introduce a directed graph (or simply digraph) to distinguish between interactions $f_{ij}$ and $f_{ji}$, which are not necessarily the same in~\eqref{eq:Model2}.  To this end, we denote the digraph by ${G_d} = (V,{E_d})$, which is  obtained from $G = (V,E)$  by replacing each undirected edge $(i,j)\in E$ with two directed edges $i\to j$ and $j\to i$. 
Since the interaction laws $f_{ij}$'s uniquely determine the vector field $f$, with a slight abuse of notations, we write 
$
{f} = (f_{ij})
$, 
where the index~$ij$ denotes an edge $i\to j$ of $G_d$.   
We further denote by ${\cal F}$ the set of any such vector fields: 
$$
{\cal F}:= \left\lbrace(f_{ij}) \mid f_{ij} \in  {\rm C}^1(\R_{\ge 0},\R)\right\rbrace,
$$ 
where ${\rm C}^1(\R_{\ge 0}, \R)$ is the set of continuously differentiable functions from $[0,\infty)$ to $\R$. For two vector fields $g = (g_{ij})$ and $h = (h_{ij})$ in ${\cal F}$, we simply let $g+ h = (g_{ij} + h_{ij}).$ We further note that any vector field $f\in {\cal F}$ gives rise to an $\SE(k)$-equivariant formation system. 

We now describe the set of allowable perturbations on the nominal dynamics~$f$.  To proceed, we first introduce the $(1,\infty)$-{\bf Sobolev norm} for a function $\psi\in {\rm C}^1(\R_{\ge 0}, \R)$: 
$$
\| \psi \|_{1,\infty} := \sup \left \{ \|\psi\|_{\infty}, \  \|\psi'\|_{\infty} \right \},
$$ where $\psi'$ is the first-order derivative of $\psi$.
We next let 
${\cal W}$ be the proper subset of  ${\rm C}^1(\R_{\ge 0},\R)$ comprised of all bounded functions with respect to the $(1,\infty)$-Sobolev norm:
$$
{\cal W}:= \left\{ \psi\in {\rm C}^1(\R_{\ge 0},\R)  \mid  \| \psi \|_{1,\,\infty} < \infty   \right\}. 
$$  
We then introduce a class of subsets of ${\cal F}$ as follows: For a subgraph $G'=(V',E')$ of $G$, and the corresponding digraph $G'_d=(V',E'_d)$, we introduce the subset
\begin{equation}\label{eq:defineFB}
{\cal H}_{G'}:= \{(h_{ij}) \mid h_{ij} \in {\cal W}, \mbox{ and } h_{ij}=0 \mbox{ if } i\to j \notin E'_d\}.
\end{equation}
In particular, if we let $G'$ be the complete subgraph $G^*$, then $\cal{H}_{G^*}$  is the set of allowable perturbations of the nominal dynamics. 
For the remainder of the section, a perturbation of the gradient vector field~$f$ is an element $h  \in {\cal H}_{G^*}$.  
We further define the  norm of $h\in {\cal H}_{G^*}$ by
$$
\| {h} \| := \max \left\{  \|h_{ij}\|_{1,\infty}      \mid  i\to j\in E^*_d   \right\}.   
$$ 
Equipped with this norm, ${\cal H}_{G^*}$ is a Banach space.


We next show that exponentially stable critical orbits of system~\eqref{eq:Model1} are robust to perturbations. To this end, recall that  $q$ is a target configuration of the formation, and the framework $(G, q)$ has a {\it full rank} clique $(G^*, q^*)$. We assume without loss of generality that $G^*$ (resp. $q^*$) is comprised of the first $(k+1)$ vertices (resp. agents) of $G$ (resp. $q$). Furthermore, we have that $O_q$ is an exponentially stable critical orbit of the potential~$\Phi$ defined in~\eqref{eq:Potential}.  
From now, we will use $f = (f_{ij})$ only when referring to the nominal gradient vector field, and $(f + h)$ a perturbed vector field, with $h\in \cal{H}_{G^*}$.  We now have the following result:

\begin{lemma}\label{lem:perturbation1}
Let $U$ be an $\SE(k)$-invariant open neighborhood of $O_q$ in $P$. Then, there is a neighborhood $\cal{U}$ of $0$ in ${\cal H}_{G^*}$ such that 
for any $h \in \cal{U}$, there is a unique $(f+h)$-invariant, exponentially stable orbit $O_{p}$ in $U$.
\end{lemma} 
The lemma follows directly from Theorem~A in~\cite{field1980equivariant}, and we thus omit the proof. We also refer to Theorem~4.1 of~\cite{hirsch2006invariant} for a similar statement in the more general context of hyperbolic invariant manifolds.

Fix a perturbation $ h\in \cal{U}$, and let $O_{p}$ be the $(f+h)$-invariant orbit whose existence is guaranteed by Lemma~\ref{lem:perturbation1}. 
 For a configuration $p'\in O_p$, we denote by $W^s_U(p'; f+h)$ the local stable manifold of $p'$ for the vector field~$(f+h)$. Shrinking the open neighborhood $U$ if necessary, we can assume that  $W^s_U(p'; f+h)$  intersects $O_q$  transversally for any $p'\in O_p$. In particular, there is a unique configuration $q'$ in the  intersection of $O_q$ and $W^s_U(p'; f+h)$.  
 Conversely, given a configuration, $q' \in O_q$, for any perturbation $h\in {\cal U}$, there is a unique configuration $p' \in O_p$ so that $q'$ lies in the intersection of $O_q$ and  $W^s_U(p'; f + h)$.

Following the arguments above, we then define the map $\rho$ which, for a target configuration $q$ fixed, assigns to a small perturbation $h$ the point $p$ in the $(f+h)$-invariant orbit so that $W^s_U(p;f+h) \cap O_q = \{q\}$:
 
\begin{equation}\label{eq:definerho}
\rho: h \mapsto p.
\end{equation}
We illustrate the definition of $\rho$ in Fig.~\ref{defrho}. Note that with the construction of~$\rho$, the $(f+h)$-invariant orbit is simply $O_{\rho(h)}$.

\begin{figure}[h]
\begin{center}
\includegraphics[width=0.6\textwidth]{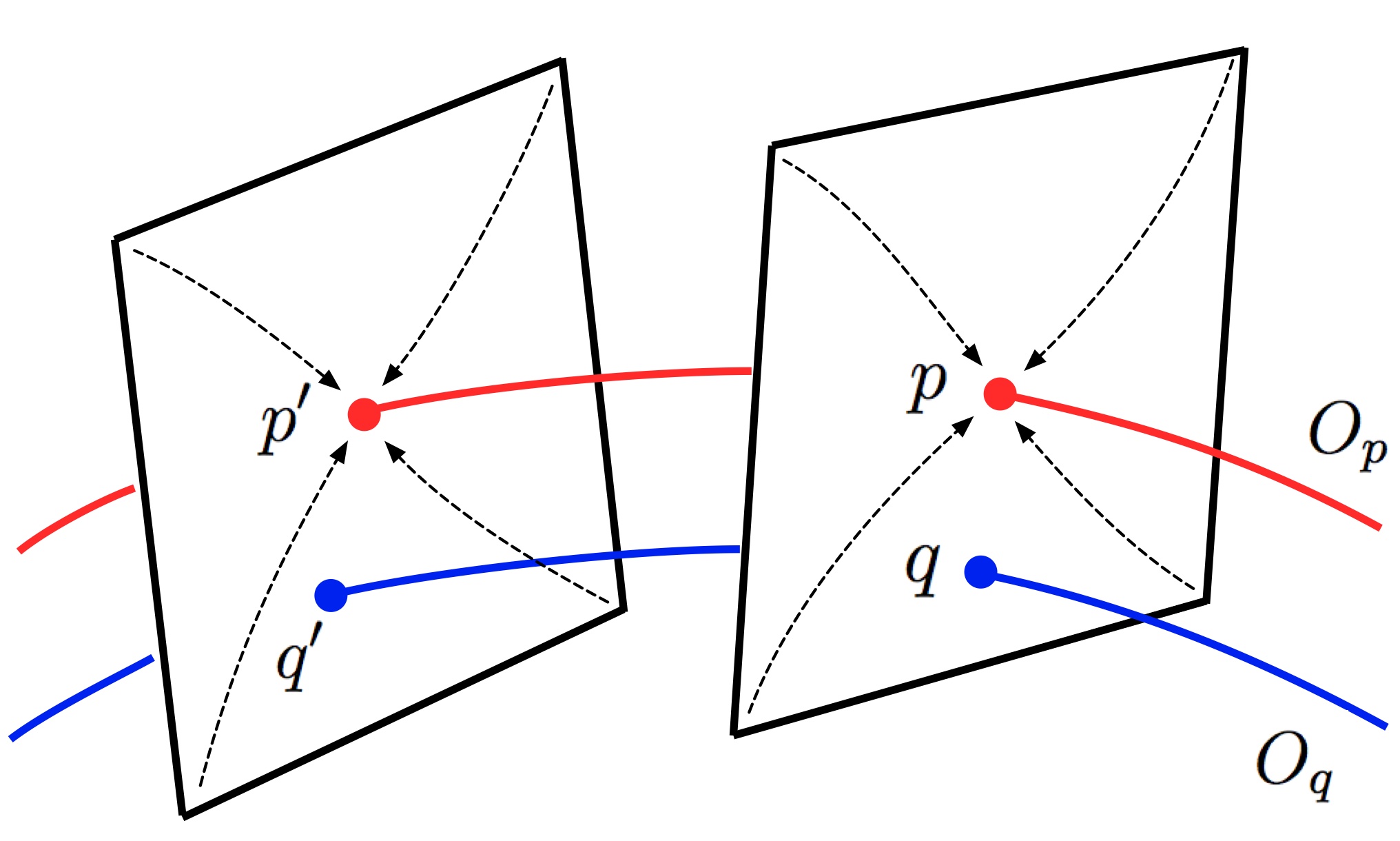}
\caption {\small  {\it  Definition of $\rho$.} The orbit $O_q$ (blue) is a critical orbit of $\Phi$ and thus comprised of target configurations. Perturbing the gradient vector field $f$ by  $h\in \cal{H}_{G^*}$,  we obtain an $(f+h)$-invariant orbit $O_p$ (red). If $\|h\|$ is sufficiently small, then $O_p$ is close to $O_q$, and moreover, the local stable manifold $W^s_U(p', f+h)$, for all $p'\in O_p$ intersects $O_q$ transversally. In particular, the intersection of $W^s_U(p', f+h)$ and $O_q$ is comprised only of a single configuration~$q'$. For a fixed target configuration $q$, there is a unique configuration $p$ such that $q\in W^s_U(p,f+h)$. The configuration $p$ depends on~$h$, and the map $\rho$ is defined by sending $h$ to $p$.}
\label{defrho}
\end{center}
\end{figure}

Recall that an equivariant vector field on an invariant orbit is completely determined by its value at a single point in the orbit, and hence an element of $\se(k)$ (see Proposition~\ref{pro:flowoninvorb}). We now fix a perturbation $h\in \cal{U}$. For convenience, we let $g:= f + h$. Since the perturbed dynamics has $O_{\rho(h)}$ as a $g$-invariant orbit by construction, there is a unique element $(\Omega_{\rho(h)}, v_{\rho(h)})\in \se(k)$ such that 
\begin{equation}\label{eq:relationpfp}
g(\rho(h)) = (\Omega_{\rho(h)}, v_{\rho(h)})\cdot \rho(h).
\end{equation} 
We further introduce the map $\omega: \cal{U} \longrightarrow \se(k)$, which assigns $h$ to the pair $(\Omega_{\rho(h)}, v_{\rho(h)})$:
\begin{equation}\label{eq:defineomega}
\omega: h \mapsto (\Omega_{\rho(h)}, v_{\rho(h)}).
\end{equation} 
With the map $\omega$,~\eqref{eq:relationpfp} is now reduced to
\begin{equation}\label{eq:relationpfphaha}
g(\rho(h)) = \omega(h)\cdot\rho(h).
\end{equation}
To summarize, to a small perturbation~$h$ of the nominal dynamics~$f$, we have assigned an $(f+h)$-invariant orbit $O_{\rho(h)}$. Moreover, the perturbed dynamics on this orbit can be written in a concise form as in~\eqref{eq:relationpfphaha}.
 
We now proceed to exhibit a relationship between the map $\omega$ and the map $\rho$, which will be of great use to show that the map $\omega$  is locally surjective---i.e., the image of $\omega$ is onto an open neighborhood of $0$ in $\se(k)$. To this end, we need to introduce Fr\'echet derivatives:

\begin{definition}[Fr\'echet derivative]
Let $X$ and $Y$ be Banach spaces, and $U$ be an open set of X. A map $\rho: U\to Y$ is said to be {\bf (Fr\'echet) differentiable} if for any $x\in U$, there is a bounded linear operator $\cal{L}_x: X \longrightarrow Y$ such that 
 $$
 \lim_{\|h\|\to 0}\frac{\|\rho(x + h) - \rho(x) - \cal{L}_x(h) \|}{\| h \|} = 0.
 $$
 We call $\cal{L}_x$  the {\bf (Fr\'echet) derivative} of $\rho$ at $x$. Further, we say that $\rho$ is {\bf continuously (Fr\'echet) differentiable} if the derivative map $$D\rho: x\mapsto \cal{L}_x$$ defined over $U$ is continuous, i.e., for any $x\in U$, 
$$ 
\lim_{\|h\|\to 0} \frac{\| \cal{L}_{x+h}(x') - \cal{L}_x(x')  \|}{\|x'\|} = 0, 
$$ 
for any $x' \in X $ with  $\|x'\| = 1$.
\end{definition}


With the preliminaries above, we have the following result: 
\begin{proposition}\label{pro:thequeenking}
Let $q$ be a target configuration of the formation system~\eqref{eq:Model1}, and $\Hess(q;f)$ be the Hessian matrix at $q$. Let $\rho$ and $\omega$ be the two maps defined in~\eqref{eq:definerho} and~\eqref{eq:defineomega}, respectively.  
Then,  $\rho$ and  $\omega$ are continuously (Fr\'echet) differentiable. Let $$d\rho_0:{\cal H}_{G^*} \longrightarrow \R^{kn} \hspace{5pt} \mbox{and} \hspace{5pt} d\omega_0: {\cal H}_{G^*}\longrightarrow \se(k)$$ be the derivatives of $\rho$ and $\omega$ at $0$, respectively. Then, for any $h \in {\cal H}_{G^*}$, the following relationship holds:
\begin{equation}\label{eq:thequeen}
h(q) = d\omega_0(h)\cdot q + \Hess(q; f) d\rho_0({h}).  
\end{equation}\,
\end{proposition}

We refer to Appendix-A for a proof of Proposition~\ref{pro:thequeenking}. 
Note that since $O_q$ is a critical orbit of system~\eqref{eq:Model1}, the Hessian matrix $\Hess(q; f)$ in~\eqref{eq:thequeen} coincides with an earlier definition~\eqref{eq:Hess1}:
$$
\Hess(q; f) = \left. -\frac{\partial f(p)}{\partial p} \right |_{p = q}= \left. \frac{\partial^2\Phi(p)}{\partial p^2} \right |_{p = q} = H_q. 
$$
It is, in particular, symmetric and has $T_qO_q$ (resp. $N^s(q)$) as its null space (resp. range space). Hence, we have 
$$
d\omega_0(h)\cdot q \in T_qO_q \hspace{10pt} \mbox{ and } \hspace{10pt} \Hess(q; f) d\rho_0(h) \in N^s(q), 
$$ 
where the first relationship holds by definition of $T_qO_q$. Thus~\eqref{eq:thequeen} provides a decomposition of the vector $h(q)\in \R^{kn}$ into components of the two orthogonal  subspaces $T_qO_q$ and $N^s(q)$.  

\section{Analysis and Proof of Theorem~\ref{thm:Main}}

\subsection{On the local surjectivity of $\omega$}\label{sec:vertshift}
We show in this subsection that the map $\omega$, which assigns to a perturbation $h$ of the nominal vector field $f$ an element  of $\se(k)$ describing the vector field $(f+h)$ at the configuration $\rho(h)$ of its $(f+h)$-invariant orbit, is {\it locally surjective}.  First, note that by definition, we have $\omega(0) = 0$. Indeed, if $h = 0$, the perturbed dynamics and the nominal dynamics are the same, but we know that the invariant orbit of the nominal dynamics consists of equilibrium points. Said otherwise, we have that $\rho(0) = q$ and $O_q$ is a critical orbit of the potential~$\Phi$. Thus, we have that $f(q) = \omega(0)\cdot q = 0$, and hence $\omega(0) = 0$. We show below that the image of $\omega$ contains an open neighborhood of $0$ in $\se(k)$, and hence $\omega$ is locally surjective around the origin.

To proceed, we recall that $G^*$ is a clique in $G$ comprised of the first $(k+1)$ vertices,  and $G^*_d = \(V, E^*_d\)$ is the corresponding complete digraph, obtained by replacing each undirected edge of $G^*$ with two directed edges. We also recall that $\overline d_{ij}$, for $(i,j)\in E$, is the target edge-length between agent $x_i$ and agent $x_j$. We further recall that a perturbation vector field ${h} = (h_{ij}) \in {\cal H}_{G^*}$ is such that  $h_{ij}\in {\cal W}$ if $i\to j \notin E^*_d$, and $0$ otherwise. We now define a map \begin{equation}\label{eq:defineeta} \eta: {\cal H}_{G^*} \longrightarrow \so(k+1): h \mapsto A_h, \end{equation} where the diagonal entries of $A_{h}$ are zeros, and the off-diagonal entries are
\begin{equation}\label{eq:defineAh}
A_{h,ij} := \frac{1}{2}\(h_{ij}\(\overline d_{ij}\) - h_{ji}\(\overline d_{ij}\)\).  
\end{equation}    
The map  $\eta$ is clearly onto $\so(k+1)$; indeed, the image of the \emph{constant} perturbations functions $h \in {\cal H}_{G^*}$ under $\eta$ is onto $\so(k+1)$.  We next establish the following result:

\begin{proposition}\label{pro:linearspan}
Let $\{ h_1,\ldots,  h_m\}$ be a subset of ${\cal H}_{G^*}$. Then, 
 $$\Sp\{d\omega_0 (h_1),\ldots,d\omega_0 (h_m)\} = \se(k)$$  
 if and only if 
$$\Sp\{\eta( h_1),\ldots, \eta( h_m)\} = \so(k+1),$$ 
with the map $\eta$ defined in~\eqref{eq:defineeta}. 
\end{proposition}

We refer to  Appendix-B for a proof of Proposition~\ref{pro:linearspan}. Since the map $\eta:{\cal H}_{G^*}\longrightarrow \so(k+1)$  is surjective, there exists a finite subset of ${\cal H}_{G^*}$: $${H} := \{h_1,\ldots,  h_{m} \},$$ with $m := k(k+1)/2$ 
such that the linear span of $\eta(h_i)$, for $i = 1,\ldots, m$, is $\so(k+1)$.
We denote by $\Sp({H})$ be the span of the elements in ${H}$: 
$$
\Sp({H}) := \left\{ \sum^m_{i = 1} c_i h_i \mid c_i \in \R \right \}.
$$
Then, $\Sp({H})$ is a finite dimensional subspace of ${\cal H}_{G^*}$. We further let 
\begin{equation}\label{eq:defineomegaF}
\omega_H: \Sp({H}) \cap \cal{U} \longrightarrow \se(k)
\end{equation} 
be the map defined by  restricting $\omega$ to the intersection of $\Sp({H})$ and $\cal{U}$, with $\cal{U}$ given in Lemma~\ref{lem:perturbation1}. We now apply Proposition~\ref{pro:linearspan} and arrive at the following result: 

\begin{corollary}\label{cor:submersion}
The map $\omega_H$ is a submersion at $0$, and thus its  image  contains an open neighborhood of $0$ in $\se(k)$.
\end{corollary}

\begin{proof} Let $d\omega_{H,0}$ be the derivative of $\omega_H$ at~$0$;  it suffices to show that it is of full rank. By definition of $\omega_H$,  $d\omega_{H,0}(h) = d\omega_0(h)$ for any $h \in\Sp({H})$. From Proposition~\ref{pro:linearspan}, we have 
$$
\Sp\{d\omega_0 (h_1),\ldots,d\omega_0 (h_m)\} = \se(k),
$$  
which completes the proof.
\end{proof}

We conclude this subsection by introducing a natural subset ${H}$ of ${\cal H}_{G^*}$, termed as the vertical shifting basis. All elements of $H$ are simply constant functions, and somehow reflect  the canonical basis of $\so(k+1)$: 

\begin{definition}[Vertical shifting basis]\label{exmp:verticalshift}
Let $G^* = \left (V^*, E^*\right )$ be the clique of $G$, comprised of the first $(k+1)$ vertices of $G$. For each edge $e = (i, j) \in E^*$ (with $i<j$), we define  $[e] = ([e]_{lm}) \in {\cal H}_{G^*}$  such that 
$$
[e]_{lm}:= 
\left\{
\begin{array}{ll}
1 & \mbox{if } l = i \mbox{ and } m = j, \\
-1 & \mbox{if } l =  j \mbox{ and } m = i, \\
0 & \mbox{otherwise}.  
\end{array}
\right.  
$$
We call the set ${H} := \left \{[e] \mid e \in E^* \right \}$ a {\bf vertical shifting basis}.   
\end{definition}

We provide below an example of vertical shifting basis for illustration:

\begin{example}
Consider the formation system depicted in Fig.~1, with the clique $G^*$ formed by the three vertices $\{1,2,3\}$. Then the vertical shifting basis in this case has three elements $[(1,2)]$, $[(1,3)]$ and $[(2,3)]$. We list below the nonzero $[e]_{lm}$'s for each $[e]$: for $[(1,2)]$, we have $[(1,2)]_{12} = 1$ and $[(1,2)]_{21} = -1$; for $[(1,3)]$, we have $[(1,3)]_{13} = 1$ and $[(1,3)]_{31} = -1$; for $[(2,3)]$, we have $[(2,3)]_{23} = 1$ and $[(2,3)]_{32} = -1$.
\end{example}

Let $H$ be the vertical shifting basis. Since $H$ is a basis of $\Sp(H)$,  for each element $h\in \Sp(H)$, we write 
$
h = \sum_{e\in E^*} c_{e} [e]
$, with $c_e$'s the real coefficients.  Let $f$ be the nominal vector field, and $g: = f + h$. Then, each $g_{ij}$ is obtained by a vertical shift of $f_{ij}$; indeed, one has that 
$
g_{ij}(x) = f_{ij}(x) \pm c_{e}
$ for all $x \ge 0$, where the plus/minus sign depends on whether $i< j$ or $i > j$. We further note that adding such a perturbation $h$ to the nominal vector field $f$ can also be interpreted as creating edge-length mismatches among the agents in the sub-configuration $q^*$.  
We consider, for example, the nominal vector field $f = (f_{ij})$ given by $f_{ij}(x) = x - \overline d_{ij}$; then, for $1\le i\neq j\le (k+1)$, 
$$
g_{ij}(x):= 
\left\{
\begin{array}{ll}
x - \(\overline d_{ij} - c_{e}\)  & \mbox{ if } i < j\\
x - \(\overline d_{ij} + c_e \)  & \mbox{ otherwise},
\end{array}
\right.
$$ 
where $c_e$ is the coefficient corresponding to the undirected edge $(i,j)$ of $G^*$.   
Thus, $\(\overline d_{ij} - c_e\)$ (resp. $\(\overline d_{ij} + c_e\)$) is the desired edge-length agent~$x_i$ (resp.~$x_j$) needs to achieve under the perturbed dynamics. An illustration is provided in Fig.~\ref{shift}.

\begin{figure}[h]
\begin{center}
\includegraphics[width=0.6\textwidth]{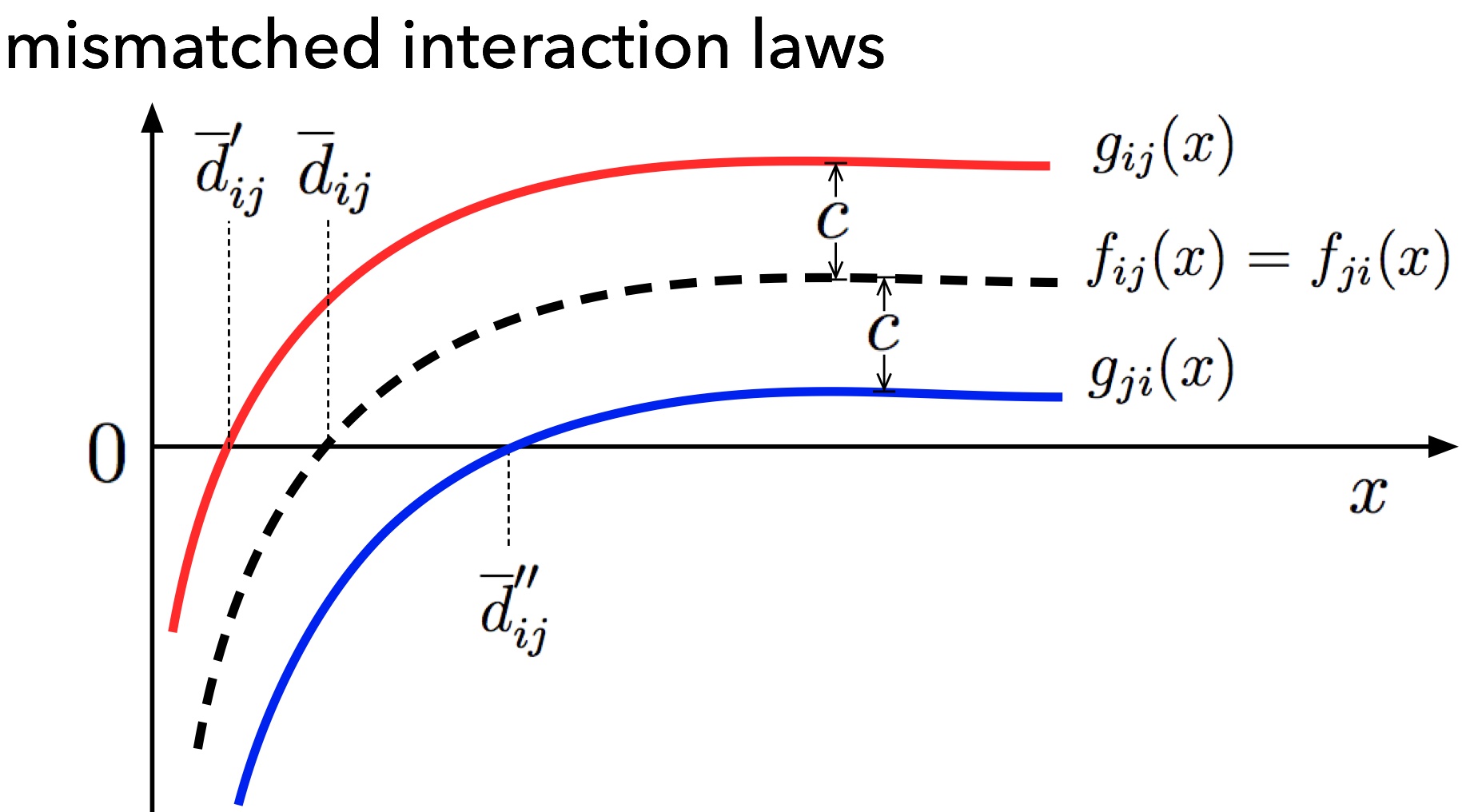}
\caption{\small Let $e = (i,j)$ be an edge of $G^*$, with $i<j$, and let $g:= f+ c[e]$ for $c > 0$. We plot $f_{ij}(x) = f_{ji}(x)$ in a black-dashed line, $g_{ij}(x)$ in red, and $g_{ji}(x)$ in blue. Note that the $f_{ij}$'s are often designed such that they are monotonically increasing, and have unique zeros at the target edge-lengths $\overline d_{ij}$'s. Then, vertical shifts of $f_{ij}$ by $c$ and $-c$, for $c$ small, gives rises to two distinct zeros, labelled as $\overline d'_{ij}$ and $\overline d''_{ij}$, respectively. Thus, $\overline d'_{ij}$ (resp. $\overline d''_{ij}$) can be viewed as the ``new'' target edge-length for~$x_i$ (resp. $x_j$) to achieve from~$x_j$ (resp. $x_i$). }
\label{shift}
\end{center}
\end{figure}


With the definition of the vertical shifting basis above, we have the following fact  as a straightforward consequence  of Proposition~\ref{pro:linearspan}.  

\begin{corollary}\label{cor:ontoeta}
Let ${H} = \left \{[e] \mid e \in E^* \right \}$ be the vertical shifting basis. Then, for an edge $e = (i,j) \in E^*$ with $i<j$, 
$$
\eta([e]) = e_ie_j^\top - e_j e_i^\top, 
$$
where $\{e_1,\ldots, e_{k+1}\}$ is the standard basis of $\R^{k+1}$. In particular, 
$$
\Sp\{\eta([e]) \mid e \in E^*\} = \so(k+1),
$$
and hence the image of 
$
\omega_H
$ 
contains an open neighborhood of $0$ in $\se(k)$.  
\end{corollary}


\subsection{Proof of Theorem~\ref{thm:Main}}
We  prove in this section that system~\eqref{eq:ControlModel} is  $\epsilon$-controllable over the orbit $O_q$ of target configurations.    
Without loss of generality, we assume that $q$ is the initial state and $\overline q\in O_q$ is the  desired final state, i.e., $q_ 0 = q$ and $q_1 = \overline q$.  
Let $ \overline \alpha = \(\overline \theta, \overline b\)\in \SE(k)$ be the (unique) element such that 
\begin{equation}\label{eq:eqeq1}
\overline q = \overline \alpha \cdot q.
\end{equation} 
Since the exponential map $\exp: \se(k)\to\SE(k)$ 
is surjective, we let $\(\overline \Omega, \overline v\)\in \se(k)$ be such that its image under $\exp(\cdot)$ is $\overline \alpha$. More specifically, from~\eqref{eq:matrixexponential}, we have that 
\begin{equation}\label{eq:thetastarvstar}
\overline \theta= \exp\(\overline \Omega\) \hspace{10pt}  \mbox{and} \hspace{10pt} \overline b= \frac{\exp\(\overline \Omega\) - I}{\overline \Omega} \overline v. 
\end{equation}
Note, in particular, that from the last item of Proposition~\ref{pro:flowoninvorb}, 
if $O_q$ were an $f$-invariant orbit with 
$
f(q) =  (r\overline \Omega, r \overline v) \cdot q
$ 
for some $r> 0$, then we would have that
$ 
\phi_{1/r}(q; f) = \overline q 
$. Said in another way, if $O_q$ were such an $f$-invariant orbit, then we would reach $\overline q$ (from $q$) in $1/r$ units of time.


Now, let ${H}$ be the vertical shifting basis introduced in  Defininition~\ref{exmp:verticalshift}. Then, from Corollary~\ref{cor:ontoeta}, the image of the map: $$\omega_H: \Sp({H})\cap \cal{U}\longrightarrow \se(k)$$ contains an open ball $B\subset \se(k)$ centered at~$0$. Because we have shown that $d\omega_H$ is of full rank at~$0$, we know from the implicit function theorem that $\omega_H$ is actually a diffeomorphism between $\Sp({H})\cap \cal{U}$ and $B$, for $\cal U$ small enough. 
Thus, for $r$ small enough so that $\(r \overline \Omega, r \overline v\) \in B$, we can set 
\begin{equation}\label{eq:hr}
h_r:= \omega_H^{-1}\(r \overline \Omega, r \overline v\) \in \Sp({H}) \cap \cal{U}.
\end{equation}
As a function of $r$, $h_r$ is continuously differentiable,  with $h_0 = 0$.

To proceed, recall that the map $\rho$ defined in~\eqref{eq:definerho} is continuously (Fr\'echet) differentiable. It sends an element $h \in \cal{U}$ to a configuration $p$, with the following properties satisfied: 
\begin{enumerate}
\item Let $g:=f + h$. Then, $O_p$ is an exponentially stable $g$-invariant orbit.
\item The local stable manifold of $p$ intersects $O_q$ transversally at $q$. 
\end{enumerate}
Now, with a slight abuse of notation, we let
\begin{equation}\label{eq:defrhor}
g_r: = f+ h_r\hspace{10pt} \mbox{and} \hspace{10pt} \rho_r:= \rho(h_r). 
\end{equation}
Because $\rho(0) = q$ and $\rho_r$ is continuously differentiable in ${r}$. So, for any $\epsilon>0$, we can shrink ${r}$ if necessary such that $\|\rho_r - q\| < \epsilon/2$.  
We can further strengthen the argument above by requiring that the two trajectories $\phi_t(q;g_r)$ and 
$\phi_t(\rho_r; g_r)$ are sufficiently close to each other. Specifically, we note that $O_{\rho_r}$ is exponentially stable for $r$ small, and moreover, $q$ lies in the stable manifold of~$\rho_r$ under~$g_r$. We thus have the following fact:


\begin{lemma}
For a given $\epsilon >0$, there exists ${r}>0$ so that the two trajectories $\phi_t(q;g_r)$ and 
$\phi_t(\rho_r; g_r)$ are $(\epsilon/2)$-close to each other:
\begin{equation}\label{eq:epsilonclose}
\|\phi_t(q;g_r) -\phi_t(\rho_r; g_r) \| < \epsilon/2, \hspace{10pt} \forall t \ge 0
\end{equation}
and
$$
\lim_{t\to\infty}\|\phi_t(q;g_r) -\phi_t(\rho_r; g_r) \|  = 0.
$$\,
\end{lemma}
In the sequel, we assume that~${r}>0$ is sufficiently small so that~\eqref{eq:epsilonclose} is satisfied. 
From~\eqref{eq:relationpfp} and~\eqref{eq:hr}, we have the following relationship:
 $$g_r = \({r} \overline \Omega, {r} \overline v\) \cdot \rho_r.$$ 
We let $\alpha(t):= \exp\(\({r} \overline \Omega, {r} \overline v\)t\)$. Note, in particular, that $\alpha(1/r) = \overline{\alpha}$ (from~\eqref{eq:matrixexponential} and~\eqref{eq:thetastarvstar}). 
Now, appealing to the last item of Proposition~\ref{pro:flowoninvorb}, we obtain that   
\begin{equation}\label{eq:eqeq2}
\phi_{t}(\rho_r; g_r) = \alpha(t) \cdot \rho_r, \hspace{10pt} \forall t \ge 0.   
\end{equation} 
On the other hand, for any $\alpha \in \SE(k)$, the affine transformation $\alpha: p \mapsto \alpha\cdot p$ is an isometry, i.e., 
for any two configurations $p$ and $p'$, we have
$$
\|\alpha\cdot p - \alpha \cdot p' \| = \|p - p'\|. 
$$
Combining this fact with~\eqref{eq:eqeq1} and~\eqref{eq:eqeq2}, we obtain that   
\begin{equation}\label{eq:halfway}
\|\phi_{t}(\rho_r; g_r) -  \alpha(t) \cdot q\|   = \|\alpha(t)\cdot \rho_r - \alpha(t) \cdot q\| 
 = \| \rho_r - q\| < \epsilon/2. 
\end{equation}
Now, appealing to~\eqref{eq:epsilonclose} and~\eqref{eq:halfway}, and the triangle inequality, we conclude that 
\begin{multline}\label{eq:epsilonclose}
d(\phi_t(q; g_r), O_q) \le 
\|\phi_t (q; g_r) -  \alpha(t) \cdot q\|  \le  \\  \|\phi_{t}(q;g_r) -\phi_{t}(\rho_r; g_r) \| 
   +  \|\phi_{t}(\rho_r; g_r) -  \alpha(t)\cdot q\|  < \epsilon.
\end{multline}
Thus, the trajectory $\phi_t(q; g_r)$ is $\epsilon$-close to $O_q$ for all $t \ge 0$. 
Furthermore, if we let $t = 1/r$ and use the fact that $\overline q =  \overline \alpha \cdot q = \alpha(1/r) \cdot q$,  then from~\eqref{eq:epsilonclose}, 
$$
\|\phi_{1/r} (q; g_r) -  \overline q\| = \|\phi_{1/r} (q; g_r) -  \alpha(1/r) \cdot q\| < \epsilon.  
$$
This then  establishes the $\epsilon$-controllability of system~\eqref{eq:Model1} over $O_q$. Indeed,  we can simply define a constant control law $u[0,1/{r}]$ as follows: for each edge $i\to j\in E^*_d$, we set
$$
u_{ij}(t):= h_{r,ij}, \hspace{10pt} \forall t\in [0,1/{r}].  
$$
Then, the trajectory generated by the formation control system~\eqref{eq:ControlModel}, with $q$ the initial condition, is exactly $\phi_t(q; g_r)$ for $t\in [0,1/{r}]$. 

We further note the following fact: If we set $u(t) = 0$ for all $t >  1/{r}$, then for $t>1/r$, the formation control system~\eqref{eq:ControlModel} is reduced to the original gradient dynamics~\eqref{eq:Model1}, where  the orbit $O_q$ of target configurations is an exponentially stable critical orbit of the associated potential function~$\Phi$. Thus, for $\epsilon$ sufficiently small, the configuration $\phi_{1/{r}}(q; g_r)$ lies in a local stable manifold $W^s_U(\widetilde q; f)$ for some $\widetilde q$ in $O_q$. Moreover, by shrinking $\epsilon$ if necessary, the distance $\|\widetilde q - \overline q\|$ can be made arbitrarily small. The arguments above thus imply the following fact as a corollary to Theorem~\ref{thm:Main}: 

\begin{corollary}\label{cor:mainthm}
Let $q, \overline q\in O_q$ be the initial and the desired final states, respectively. Then, for any error tolerance $\epsilon > 0$,  there is a time $T > 0$ and a constant control law $u[0,T]$, which steers the formation control system~\eqref{eq:ControlModel} along a trajectory $q(t)$, with $q(0) = q$, over the time period $[0, T]$ such that the following hold:
\begin{enumerate}
\item For any $t\in [0, T]$, we have $d(q(t), O_q) < \epsilon$. Moreover, $\|q(T) - \overline q\| < \epsilon$. 
\item Let $q'(t)$, with $q'(0) = q(T)$,  be the trajectory generated by the gradient system~\eqref{eq:Model1} (or equivalently, we set $u([T,\infty]) = 0$ for~\eqref{eq:ControlModel}), then $q'(t)$ converges exponentially fast to a configuration $\widetilde q$ in $O_q$. Moreover, $\|\widetilde q - \overline q\| < \epsilon$.
\end{enumerate}   
\end{corollary}

We further provide in Fig.~\ref{trajectory} an illustration of Theorem~\ref{thm:Main} and the corollary above.
\begin{figure}[h]
\begin{center}
\includegraphics[width=0.6\textwidth]{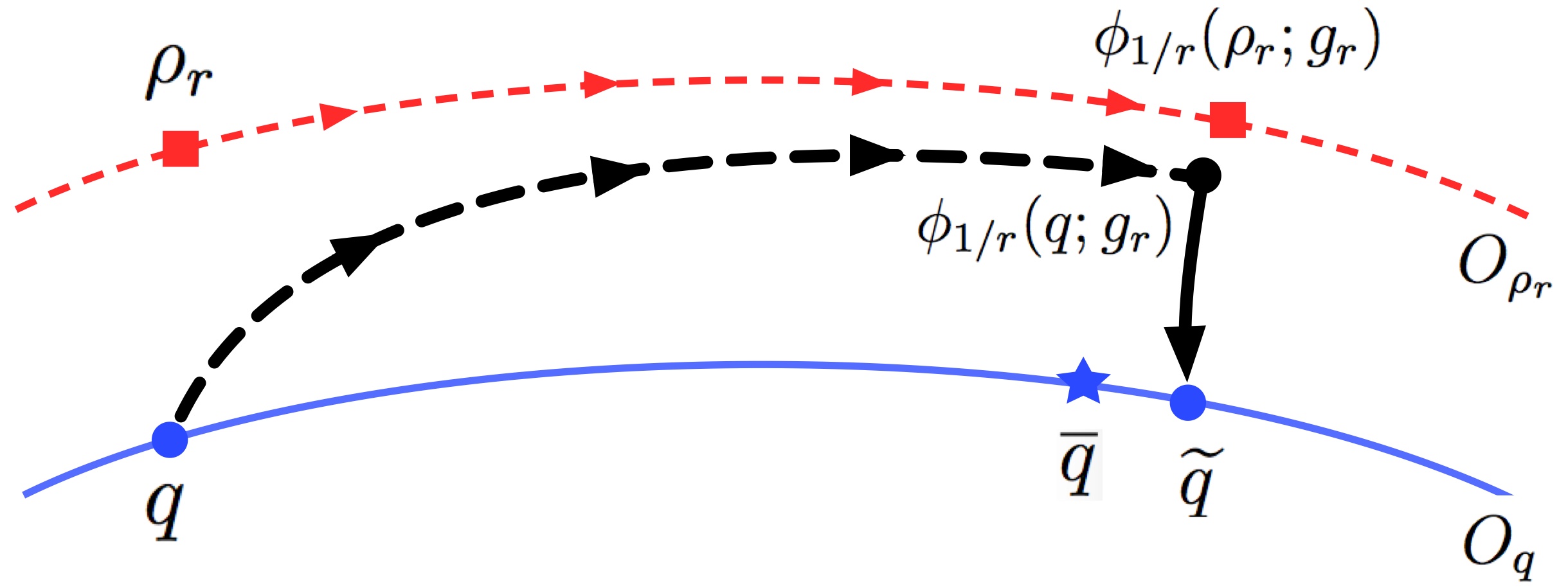}
\caption {\small {\it Illustration of Theorem~\ref{thm:Main} and Corollary~\ref{cor:mainthm}.} Let $q$ be the initial state and $\overline q = \overline \alpha \cdot q$  the desired final state. The vector field $g_r$ and the configuration $\rho_r$ are defined in~\eqref{eq:defrhor}. For $r$ sufficiently small, $O_{\rho_r}$ is an exponentially stable $g_r$-invariant orbit, with $g_r(\rho_r) = \(r\overline \Omega, r\overline v\)\cdot \rho_r$.  The red-dashed (resp. black-dashed) segment with arrows represents the trajectory $\phi_{t}(\rho_r; g_r)$ (resp. $\phi_t(q;g_r)$), for $t\in [0,1/r]$.  By our earlier construction, $q$ is in the stable manifold of $\rho_r$ under $g_r$, and hence $\phi_t(q; g_r)$ converges to $\phi_t(\rho_r; g_r)$ as $t\to\infty$.  At the time $t  = 1/r$,  we have that $\phi_{1/r}(\rho_r; g_r) = \overline \alpha \cdot \rho_r$. Note, in particular, that both $\overline q$ and $\phi_{1/r}(q; g_r)$ are in the stable manifold of $\overline \alpha \cdot \rho_r$ under $g_r$, and the two configurations are $\epsilon$-close to each other.  For $t \ge 1/r$, we set $u = 0$, and hence our formation system~\eqref{eq:ControlModel} is reduced to the original gradient dynamics~\eqref{eq:Model1}. Thus, for trajectory generated by the gradient system, with $\phi_{1/r}(q; g_r)$ the initial condition, converges to a configuration~$\widetilde q$ in $O_q$. Such a trajectory is plotted using a black-solid segment with an arrow. Moreover, the configuration $\widetilde q$ can be made arbitrarily close to $\overline q$ by shrinking $r$, and hence  the distance between $\phi_{1/r}(q; g_r)$ and $\overline q$.}
\label{trajectory}
\end{center}
\end{figure}

\section{On the robustness issue in formation control}\label{sec:robustnessissue}
In the previous section, we have shown that the formation control system~\eqref{eq:ControlModel} can be   steered by controlling the interactions among  $(k+1)$ agents, provided these agents form a nondegenerate $k$-simplex in the formation. In this section, we demonstrate that this approach can lead to a practical scheme to remedy the robustness issue  of system~\eqref{eq:Model1}  highlighted  in~\cite{mou2016undirected}.


We first briefly describe the  cause of the robustness issue that arises: Gradient control laws, widely used in formation control, yield reciprocal interactions between agents, i.e.,   
\begin{equation}\label{eq:assumption1}
f_{ij} = f_{ji}, \hspace{10pt} \forall (i,j)\in E.
\end{equation} 
In practice, the finite precision of our measurements and actuation  prevents the implementation of such a perfect reciprocity.  This lack of exact reciprocity, which in typical control scenarios results in the stabilization of the system at a nearby equilibrium, was shown to lead in the case of formation control to a rigid, non-trivial, motion of the formation. This periodic motion could be  of potentially large radius. 

{\color{black} We represent the lack of reciprocity by a perturbed dynamics $(f + h)$, with $f = (f_{ij})$ satisfying~\eqref{eq:assumption1}, and $h \in {\cal H}_G$. } Unlike the case we dealt with in the previous section, where $h$ is in ${\cal H}_{G^*}$,  the perturbation $h$ here is in $\cal{H}_{G}$. Said in another way, each $h_{ij}$, for $i\to j\in E_d$, can be nonzero.  Denote by $g:= f + h$  the vector field of the generalized formation system~\eqref{eq:Model2}.  From  arguments similar to the ones used  in Subsection~\ref{sec:perturba}, we have that for $\|h\|$ sufficiently small, there is  a unique exponentially stable $g$-invariant orbit $O_p$ near the critical orbit $O_q$ of the unperturbed system. The robustness issue is a consequence of the fact that the vector field $g$, when restricted to the orbit~$O_p$, does not necessarily vanish. Hence, for any configuration $p'\in O_p$, the formation is not at rest,  but  drifts within the orbit $O_p$, with the trajectory described by~\eqref{eq:soloninvariantrobit} in Proposition~\ref{pro:flowoninvorb}. We also refer to~\cite{sun2013non} for simulations of such a drift in lower dimensional cases where $k = 2, 3$.  

We now show how the formation control method described in~\eqref{eq:ControlModel} can be used to fix the robustness issue. Roughly speaking, we show that for any perturbation $h$ with $\|h\|$ sufficiently small, there exists a   \emph{constant} control law $u$ for~\eqref{eq:ControlModel} that {\it offsets} the drift  caused by the perturbation.  When equipped with this offset, the system will stabilize an orbit $O_p$  close to the orbit $O_q$ of target configurations. This  orbit $O_p$, moreover,  contains only equilibrium points and is \emph{exponentially} stable.  As a result, the compensating offset $u$ will, in the presence of the perturbation~$h$, stabilize the system at inter-agent distances slightly different from the target edge-lengths, and prevent a translation and/or rotation of the formation.

The offset $u$ certainly depends on the perturbation $h$. While the perturbation $h$ affects the nominal dynamics at {\it all} agents, the offset $u$ is  only applied to the agents designated by a clique $G^*$ in the formation graph. Hence, in order to fully resolve the robustness issue using the approach of this paper, one needs to close the loop by developing a method which assigns a proper offset $u$ from the measurements available to the agents of the clique $G^*$. 
We believe that this is feasible, but is outside the scope of this paper. For the remainder of the paper, we solely focus on showing that for any global perturbation $h$, global in the sense that it applies to the interactions among all neighboring agents, there exists a local  compensating offset $u$, local in the sense that it applies only to the interactions among the agents of the clique.

We now state  the result in precise terms. Recall that $f_i$ is the $i$-th component of $f$ as defined in the earlier sections. Similarly, we let $g_i$ be the $i$-th component of the perturbed dynamics $g = f + h$, for $h\in {\cal H}_G$. The formation control model~\eqref{eq:ControlModel} now takes the following form: 
\begin{equation}\label{eq:ControlModel2}
\dot {x}_i = 
\left\{
\begin{array}{ll}
g_i(p)+ \sum_{j \in \cal{N}^*_i } u_{ij} (x_j - x_i)   & \text{ if } i\in V^* \\
g_i(p) & \text{ otherwise},
\end{array} \right.
\end{equation}  
with the gradient vector field $f$ replaced by the perturbed dynamics $g$. 
We still let $u = (u_{ij})$ be the ensemble of the $u_{ij}$'s, which itself can be treated as an element in ${\cal F}$.  
We establish below the main result of the section: 

\begin{theorem}\label{thm:robustness}
There is an open neighborhood $\cal{U}$ of $0$ in ${\cal H}_G$ and an $\SE(k)$-invariant open neighborhood $U$ of $O_q$ in $P$ such that  for any perturbation $h \in \cal{U}$, there is a  compensating offset ${u}$ for which the control system~\eqref{eq:ControlModel2} possesses a  unique invariant orbit $O_p$ in $U$, which is moreover comprised only of equilibrium points. The orbit $O_p$ is  exponentially stable.   
\end{theorem}  

The proof of Theorem~\ref{thm:robustness} relies on the  following two results:

\begin{lemma}[Differentiable slice theorem]\label{lem:diffslicethm}
Let $\cal{A}$ be a Lie group of dimension $m$ acting on a Euclidean space $P$. Let $O_p:= \cal{A}\cdot p$ be an orbit through $p\in P$. Assume that the stabilizer of $p$ is trivial and thus $\dim O_p=m$.  Then, there is a submanifold $S$ of codimension $m$ in $P$ intersecting $O_p$ transversally at $p$, and an embedding $\iota: S\times O_p \to P$ mapping  $S\times O_p$ diffeomorphically onto an $\cal{A}$-invariant open neighborhood $U$ of $O_p$ in $P$.  Moreover, the map $\iota$ satisfies  the condition that 
\begin{equation}\label{eq:differentialslice}
\iota(s, \alpha\cdot p) = \alpha\cdot \iota(s,p), \hspace{10pt} \forall \alpha\in \cal {A}. 
\end{equation}
Any set $S$ for which there exists an embedding $\iota$ satisfying the above conditions is called a {\bf slice} at $p$.
\end{lemma}

The differentiable slice theorem is about a decomposition of an open neighborhood $U$ of an orbit $O_p$ into a product of two transversal submanifolds $S$ and $O_p$. Given a slice $S$, the map $\iota$ can be constructed explicitly as follows: Since the stabilizer of $p$ is trivial, for a given $p'\in O_p$, there is a unique $\alpha\in \cal{A}$ such that $p' = \alpha\cdot p$ (and hence $O_p\approx \cal{A}$). We then define $\iota$ by sending $(s, p')$ to $\alpha\cdot s$. We note that the choice of a slice may not be unique. We further illustrate the differentiable slice theorem in Fig.~\ref{slice}.

\begin{figure}[h]
\begin{center}
\includegraphics[width=0.6\textwidth]{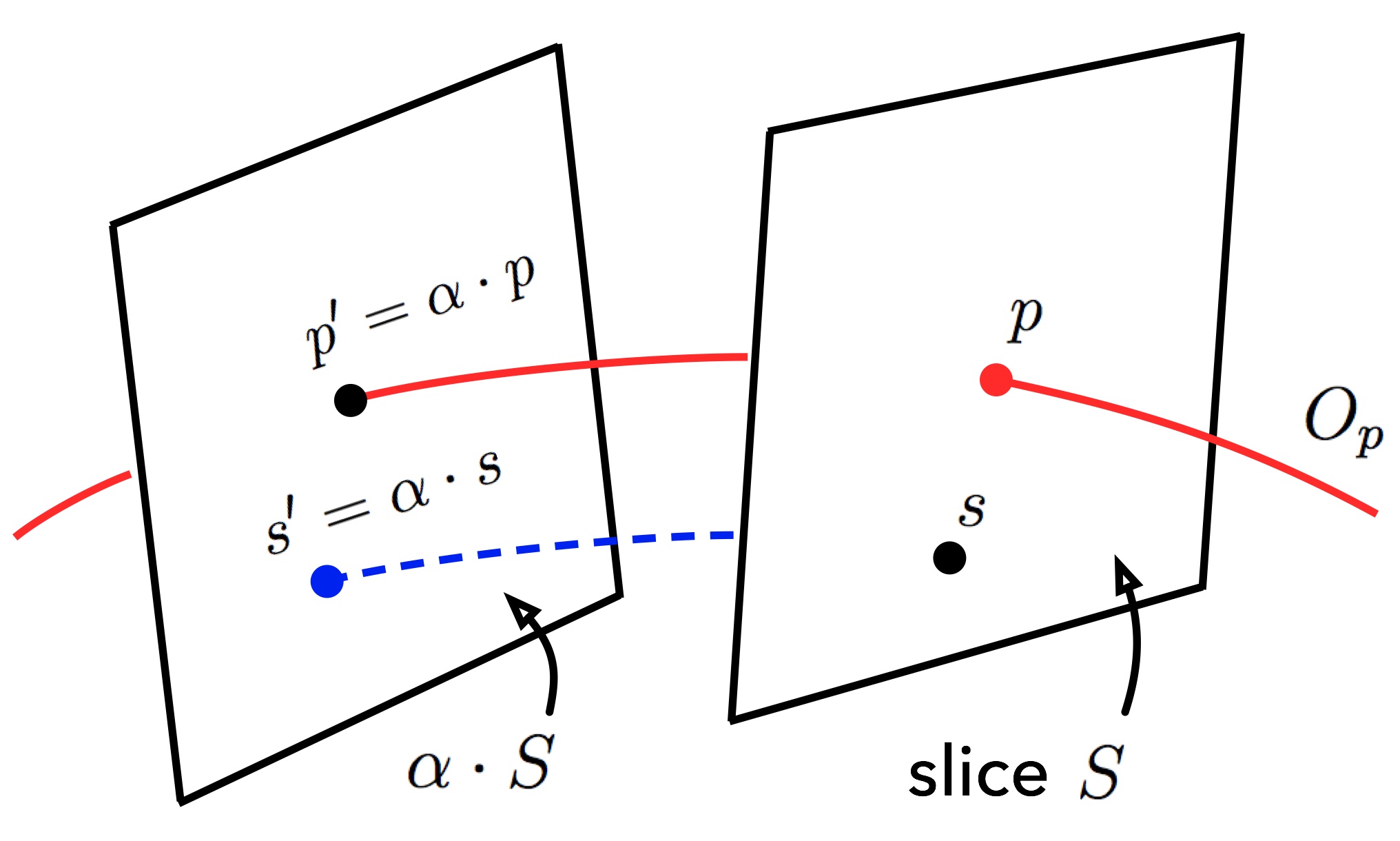}
\caption{ \small {\it Illustration of the differentiable slice theorem. } The submanifold $S$ intersects $O_p$ transversally, with $p$ the unique intersection point. The diffeomorphism~$\iota$ is defined as follows: for a given pair $(s,  p')$, with $s\in S$ and $p' = \alpha\cdot p\in O_p$ (both marked in black), we set $\iota(s, p') := \alpha\cdot s$ (blue), which lies in $\alpha\cdot S$.   }
\label{slice}
\end{center}
\end{figure}

In our case, we have that $\SE(k)$ acts on $P$, and the orbit is $O_q$, with $q$ a target configuration. Since $q$ is of full-rank, from Lemma~\ref{lem:trivialstabilizer}, the stabilizer of $q$ is trivial. Moreover, since $O_q$ is exponentially stable, we have a particular choice of the slice $S$ as stated below:

\begin{lemma}
Let $U$ be an $\SE(k)$-invariant open neighborhood of $O_q$ in $P$. Then, for $U$ sufficiently small, the local stable manifold $W^s_U(q;f)$ is  a slice at~$q$. 
\end{lemma}

\begin{proof}
Let $U$ be an $\SE(k)$-invariant open neighborhood $U$ of $O_q$ such that Lemma~\ref{lem:fibersmoothsubmfld} is satisfied (with $p$ replaced by $q$). Recall that the local stable manifold~$W^s_U(q; f)$ intersects $O_q$ transversally at $q$, with $q$ the only intersection point.  We then define an embedding $$\iota:W^s_U(q;f)\times O_q\longrightarrow U: (s, \alpha\cdot q)\mapsto\alpha \cdot s,$$ 
which is a diffeomorphism satisfying~\eqref{eq:differentialslice}. \end{proof}

If the open neighborhood $U$ is small, then each orbit in $U$ intersects $W^s_U(q; f)$, and moreover, the intersection is transversal. This, in particular, implies that their intersection  contains {\em only} one point. Thus, each configuration $p$ in $W^s_U(q; f)$ can be used to  represent an orbit $O_{p}$ in $U$.  For ease of notation, we  set  
 $S:= W^s_U(q; f)$.

The second result required  is a variation on the implicit function theorem. Recall that $\Sp({H})$ is a finite dimensional vector space spanned by the vertical shifting basis~${H}= \{[e] \mid e \in E^*\}$ (see Definition~\ref{exmp:verticalshift}). We consider below {\em constant} control laws~$u$ in $\Sp(H)$, i.e., $u = \sum_{e\in E^*} c_e [e]$,  where the $c_e$'s are the real coefficients. We note here that the Banach space $\cal{H}_{G}$ should be considered as the collection of all possible perturbations~$h$ of~$f$, while the finite-dimensional space~$\Sp(H)$ is  the collection of the constant control laws (the offsets) we apply to remedy the robustness issue. To this end,    
we introduce a function $\xi$ as follows:  $$\xi:{\cal H}_G\times \Sp({H}) \times S \longrightarrow \R^{kn}:(h,u,p) \mapsto f(p)+h(p)+u(p).$$ 
The map~$\xi$ takes into account both the perturbation~$h$ and the control~$u$, and evaluates the vector field of the control system~\eqref{eq:ControlModel2} at a given point $p\in S$. 
The map $\xi$ is continuously (Fr\'echet) differentiable since it is linear. We also note that if $\xi(h,u, p) = 0$ for a given triplet $(h, u, p)$, then $\xi(h, u, p') = 0$ for all $p'\in O_p$, and hence $O_p$ is orbit of system~\eqref{eq:Model2} comprised of equilibrium points. Thus, it suffices to establish the  following fact about the existence of a local offset $u$ which stabilizes an orbit $O_p$ of the perturbed (or mismatched) formation control system:

\begin{proposition}\label{pro:implicitfunction}
There exists an open neighborhood $\cal{U}$ of $0$ in ${\cal H}_G$, an open neighborhood $\cal{V}$ of $(0, q)$ in $\Sp({H}) \times S$, and a unique continuously  (Fr\'echet) differentiable map   $\zeta: \cal{U} \longrightarrow \cal{V} $  
such that  
$
\xi(h, \zeta(h)) = 0 
$ for any $h\in \cal{U}$. 
\end{proposition}

We refer to the Appendix-C for a proof of Proposition~\ref{pro:implicitfunction}. With Proposition~\ref{pro:implicitfunction}, we are now in a position to prove Theorem~\ref{thm:robustness}

\begin{proof}[Proof of Theorem~\ref{thm:robustness}] 
The proof follows from Proposition~\ref{pro:implicitfunction} and Lemma~\ref{lem:perturbation1}. Indeed, from Proposition~\ref{pro:implicitfunction}, we know that  for $\|h\|$ sufficiently small, there exists a unique constant control law $u$ for which the control system~\eqref{eq:ControlModel2} possesses a unique orbit $O_p$, with $p\in S$, comprised of equilibrium points.  Moreover, the map  
$\zeta: h \mapsto (u, p)$ is continuously (Fr\'echet) differentiable. 
For convenience, we let $g := f + h + u$ be the vector field of the control system~\eqref{eq:ControlModel2}. By the continuity of the map~$\zeta$, 
we have that if $\|h\|$ is small, then so is $\|u\|$, and hence $\|g - f\|$. Thus, from Lemma~\ref{lem:perturbation1}, we conclude that if $\|h\|$ is sufficiently small, then $O_p$ is an exponentially stable $g$-invariant orbit comprised only of equilibrium points. This completes the proof.  
\end{proof}

\section{Conclusion}

This paper has dealt with the problem of controlling a formation of $n$ agents in a Euclidean space. The control objective was to steer a rigid formation to an arbitrary position and orientation in the Euclidean space. The type of formations handled are the ones in which the agents aim to stabilize at given inter-agent distances, called target edge-lengths, and rely on feedback control laws to achieve this goal. The agents are thought of as vertices in a graph that describes both the target edge-lengths and the decentralization structure of the feedback (an agent has access to the relative position of agents related to him by an edge in the graph). The main message of the paper was that if a subset of the agents can freely modify the target edge-lengths that are assigned to them, and if this subset of agents contains a clique in the formation graph, then they can control the rigid motion of the entire formation.   The proof relied on showing that by appropriately choosing the target edge-lengths of said subset of agents, we could induce an arbitrary rigid motion for the formation as a whole. 
Besides its use in controlling formations, we believe that this result may help design control laws that fix the robustness issue described in~\cite{mou2016undirected}.


\bibliographystyle{IEEEtran}
\bibliography{controlformationtriangle-1}

\section*{Appendix}
\setcounter{subsection}{0}

We prove here Propositions~\ref{pro:thequeenking}, ~\ref{pro:linearspan}, and~\ref{pro:implicitfunction}. 

\subsection{Proof of Proposition~\ref{pro:thequeenking}}
The fact that $\rho$ and $\omega$ are continuously differentiable follows from Theorem~4.1 of~\cite{hirsch2006invariant}. We establish here~\eqref{eq:thequeen}. Let $\h$ be an element of ${\cal H}_{G^*}$. We perturb $\f$ and have  $(\f + \epsilon \h)$ for $\epsilon$ small. Then, from~\eqref{eq:relationpfphaha}, we have
\begin{equation}\label{eq:firstorderanalysis}
(f + \epsilon h)(\rho(\epsilon\h)) = \omega(\epsilon \h) \cdot \rho(\epsilon\h).
\end{equation}   
Note that $\rho(0) = q$, $\omega(0) = 0$ and $f(q) = 0$ (since $q$ is a critical point of the potential $\Phi$). Thus, up to the first order of $\epsilon$, we obtain from~\eqref{eq:firstorderanalysis} the following relation:
\begin{equation}\label{eq:assumptionone}
\left. \frac{\partial f(p)}{\partial p} \right |_{p = q} d\rho_0(\h) + h(q) = d\omega_0(\h)\cdot q.
\end{equation}
Since $f$ is a gradient vector field, we have 
\begin{equation}\label{eq:assumptiontwo}
\Hess(q; f) = -\left. \frac{\partial f(p)}{\partial p}\right |_{p = q}.
\end{equation}
Combining~\eqref{eq:assumptionone} and~\eqref{eq:assumptiontwo} leads to~\eqref{eq:thequeen}. \hfill$\blacksquare$

\subsection{Proof of Proposition~\ref{pro:linearspan}}
Let   $\h: = \sum^m_{i=1} c_i \h_i$, with $c_1,\ldots, c_m$ arbitrary real numbers. Because $\se(k)$ and  $\so(k+1)$ have the same dimension.  So, it suffices to show that
\begin{equation}\label{eq:ifandonlyifcondition}
 d\omega_0(\h) = 0\ \Longleftrightarrow\  \eta(\h) = 0.
\end{equation}
From~\eqref{eq:thequeen}, we have that for any $y\in T_qO_q$,  
\begin{equation}\label{eq:innerproductcopypre}
\langle y, h(q) \rangle = \langle y, d\omega_0(\h)\cdot q + \Hess(q; f) d\rho_0(\h) \rangle, 
\end{equation}
where $\langle \cdot, \cdot \rangle$ is the standard inner-product in $\R^{kn}$. 
Note that $\Hess(q; f) d\rho_0(\h)$ lies in $N^s(q)$, i.e., the range space of $\Hess(q; f)$, which is  perpendicular to $T_qO_q$ because $\Hess(q; f)$ is symmetric. Thus, we have 
$$
\langle y, \Hess(q; f) d\rho_0(\h) \rangle = 0, \hspace{10pt} \forall y \in T_qO_q.
$$
Thus,~\eqref{eq:innerproductcopypre} is reduced to
\begin{equation}\label{eq:innerproductcopy}
\langle y, h(q) \rangle = \langle y, d\omega_0(\h)\cdot q \rangle, \hspace{10pt} \forall y \in T_qO_q.
\end{equation}
Since $q$ is of full rank,  
\begin{equation*}
d\omega_0(\h) = 0 \ \Longleftrightarrow\  d\omega_0(\h)\cdot q  = 0. 
\end{equation*}   
On the other hand, $d\omega_0(\h)\cdot q$ lies in $T_qO_q$, 
and hence, from~\eqref{eq:innerproductcopy}, we have
$$
d\omega_0(\h) = 0 \ \Longleftrightarrow\ \langle y, h(q) \rangle = 0, \hspace{10pt} \forall y\in T_qO_q.
$$
From~\eqref{eq:ifandonlyifcondition}, it thus suffices to show that  
\begin{equation}\label{eq:charles}
 \eta(\h) = 0  \ \Longleftrightarrow\ \langle y, h(q) \rangle = 0, \hspace{10pt} \forall y\in T_qO_q.
\end{equation}
To evaluate the inner product $\langle y, h(q) \rangle$, we first write $y = (\Omega, v)\cdot q$ for some $(\Omega, v)\in \se(k)$. Note that such a pair $(\Omega, v)$ is unique since $q$ is of full-rank. Thus, if $\langle y, h(q) \rangle = 0$ for all $y\in T_qO_q$, then $\langle (\Omega, v) \cdot q, h(q) \rangle = 0$ for all $(\Omega, v)\in \se(k)$ and vice versa. 
To proceed, we recall that the $(k+1)\times (k+1)$ matrix $X$ in the proof of Lemma~\ref{lem:trivialstabilizer} is defined as follows: the $i$-th column of $X$ is $(x_i,1)\in \R^{k+1}$. Since the first $(k+1)$ agents of $q$ form a $k$-simplex of full rank,  the matrix $X$ is nonsingular. 
Now, define a linear isomorphism $T: \se(k) \longrightarrow \so(k+1)$ as follows: 
$$
T:  (\Omega, v)\in \se(k) \mapsto 
\begin{bmatrix}
\Omega & v\\
-v^\top & 0
\end{bmatrix}\in\so(k+1). 
$$
Then, by computation, we obtain
$$
\langle (\Omega, v)\cdot q, h(q) \rangle = \tr\left(T(\Omega, v) X \eta(\h) X^\top \right).
$$
Thus, if $\langle (\Omega, v)\cdot q, h(q)\rangle$ vanishes for all $(\Omega, v)\in \se(k)$, then 
\begin{equation}\label{eq:olOmega}
\tr\left(\widetilde\Omega\, X \eta(\h) X^\top \right) = 0, \hspace{10pt} \forall \ \widetilde\Omega\in \so(k+1).
\end{equation}
Since $X$ is nonsingular and $\eta(\h)$ is skew-symmetric, we conclude that \eqref{eq:olOmega} holds if and only if $\eta(\h)  = 0$.
\hfill{\qed}

\subsection{Proof of Proposition~\ref{pro:implicitfunction}}
We first let $\tau := (0,0, q) \in {\cal H}_G\times \Sp({H}) \times S$. We then let $\partial_{1} \xi_\tau$, $\partial_2\xi_\tau$, and $\partial_3\xi_\tau$ be the partial derivatives of $\xi$ at~$\tau$, with respect to the arguments $\h$, $u$ and $p$. Note that the dimension of $\Sp({H})$ is $k(k+1)/2$, and the codimension of~$S$ in~$P$ is also $k(k+1)/2$. It then follows that 
$$\dim\left (\Sp({H})\times S\right ) = kn.$$  
From the implicit function theorem, it suffices for us to prove that the map
$$
\partial_2\xi_\tau \oplus \partial_3\xi_\tau: \Sp({H})\oplus  T_qS \longrightarrow \R^{kn}
$$
defined by 
$$
\partial_2\xi_\tau \oplus \partial_3\xi_\tau: (u, v) \longrightarrow \partial_2\xi_\tau(u) + \partial_3\xi_\tau(v)
$$
is a linear isomorphism. To establish this fact, we first compute the range spaces of $\partial_i\xi_\tau$, denoted by $R_i$, for $i = 2, 3$. In particular, we show that 
$$\dim R_2 = \operatorname{codim} R_3 = \frac{1}{2}k(k+1).$$  We then  show that $R_2\cap R_3 = \{0\}$. It then follows that $R_2 \oplus R_3 =\R^{kn}$, and hence $\partial_2\xi_\tau \oplus \partial_3\xi_\tau$ is a linear isomorphism.  

We first compute the range spaces $R_i$, for $i = 2, 3$. For $R_2$, we note that the map $\xi$ is affine when restricted to $0\times \Sp({H})\times \{q\}$. Specifically, we have 
$
\xi(0, u, q) = u(q) + f(q)
$, 
and hence 
$
\partial_2\xi_\tau(u) = u(q)
$. Thus, $R_2$ is given by 
\begin{equation}\label{eq:rangespace2}
R_2= \Sp\{u(q)\mid u\in H\}.
\end{equation}
Since the sub-configuration $q^*$ formed by the first $(k+1)$ agents is nondegenerate, by our construction of the vertical shifting basis, $u(q) = 0$ if and only if $u = 0$. We thus have $\dim R_2 = k(k+1)/2$. For $R_3$,  
we recall that $S$ is the local stable manifold $W^s_U(q; f)$, and the tangent space of $W^s_U(q; f)$ at~$q$ is the range space of the Hessian matrix $\Hess(q; f)$ (see Remark~\ref{rmk:definitionofNs}). Since $f$ is the gradient vector field of the potential~$\Phi$, one has that 
$$\Hess(q; f) = - \frac{\partial f(q)}{\partial p}  =  \frac{\partial^2\Phi(q)}{\partial p^2},$$
which is symmetric.  
Since $O_q$ is exponentially stable (and hence nondegenerate), the null space of $\Hess(q; f)$ is $T_qO_q$. Thus,  the range space $N^s(q)$ of $\Hess(q; f)$ is the subspace perpendicular to $T_qO_q$. By the arguments above, we obtain that 
$$
R_3 = \partial_3\xi_\tau \(T_qS\) =   \frac{\partial f(q)}{\partial p} \(N^s(q)\) =  N^s(q), 
$$
 and hence $\operatorname{codim}R_3 = \dim T_qO_q = k(k+1)/2$. 
 
We now prove that $R_2 \cap R_3 = \{0\}$. Let $u\in \Sp(H)$; it suffices to show that if $u(q) \in N^s(q)$, then $u = 0$.  Since $N^s(q)$ is perpendicular to $T_qO_q$, we have 
$\langle u(q), y\rangle = 0$ for all $y \in T_qO_q$. 
Then, from~\eqref{eq:charles}, we obtain that $\eta(u) = 0$, where the map $\eta$ is defined in~\eqref{eq:defineeta}. On the other hand, we know from Corollary~\ref{cor:ontoeta} that 
$
\{\eta([e]) \mid e\in E^*\} 
$ is a basis of $\so(k+1)$. Since $u$ is a linear combination of the $[e]$'s and the map $\eta$ is linear, we have that $\eta(u)$ is a linear combination of the $\eta([e])$'s. We thus conclude that $\eta(u) = 0$ if and only if $u = 0$. This completes the proof. 
\hfill{\qed}

\end{document}